\documentclass[a4paper, 12pt, reqno]{amsart}
\usepackage{amssymb, amsfonts, amsmath, bm}
\usepackage[all]{xy}
\usepackage[shortlabels]{enumitem}
\usepackage{hyperref, color}
\usepackage{commath}
\usepackage{esdiff}
\usepackage{graphicx,subfigure}
\usepackage[center]{caption}
\usepackage{amsthm}
\usepackage{diagbox}
\usepackage{multirow}
\usepackage{esdiff}
\usepackage[pagewise]{lineno}
\usepackage[utf8]{inputenc}
\usepackage{soul}
\usepackage{comment} 
\usepackage{pifont}
\usepackage{todonotes}
\usepackage{hyperref, color, xcolor}
\hypersetup{
	colorlinks = true,
	linkcolor = blue,
	filecolor = blue,
	urlcolor = red,
	citecolor = red
}
\usepackage{cite}
\usepackage{lineno}

\usepackage{dirtytalk}

\usepackage[nameinlink,capitalise]{cleveref} 

\usepackage[a4paper,twoside,top=1in, bottom=1in, left=0.8in, right=0.8in]{geometry}

\DeclareMathOperator{\Ric}{Ric}

\newtheorem{theorem}{Theorem}[]
\newtheorem{lemma}[theorem]{Lemma}
\newtheorem{proposition}[theorem]{Proposition}
\newtheorem{corollary}[theorem]{Corollary}
\newtheorem*{conjecture*}{Conjecture}

\newtheorem{theo}{Theorem}

\newtheorem{cor}[theo]{Corollary}

\theoremstyle{definition}

\newtheorem*{remark*}{Remark}

\newtheorem{remark}[theorem]{Remark}

\numberwithin{figure}{section}

\usepackage{graphicx}
\colorlet{BLUE}{blue} 

\makeatother
\title[Integral Identities and Rigidity of $m$-Quasi-Einstein Manifolds]{Integral Identities and Rigidity of Generalized $m$-Quasi-Einstein Manifolds}
\author[Alcides de Carvalho and W. O. Costa-Filho]{Alcides de Carvalho and W. O. Costa-Filho}

        \address{Universidade Federal de Pernambuco\\
		Departamento de Matematica\\
		 Av. Jorn. Aníbal Fernandes, s/n, Cidade Universitária, Recife - PE, Brazil, CEP:  50740-560.}	
	\email{alcides.junior@ufpe.br}
    
\address{
  Universidade Federal de Alagoas, Campus Arapiraca, Av. Manoel Severino Barbosa, S/N, Bom Sucesso, 
 Arapiraca - AL, Brazil, CEP: 57309-005.}
\email{wagner.filho@arapiraca.ufal.br}

\date{\today}

\begin{document}



\begin{abstract}
We prove that every closed $m$-quasi-Einstein manifold
$(M^n,g,X,\lambda)$ with constant $\lambda\leq0$ is trivial whenever
$m\leq-2$. This establishes the Colling--Dunajski conjecture in the range
$m\leq-2$ and, for $n\geq5$, extends the previously known range
$m\leq2-n$. 
This result is placed within a broader study of closed generalized
$m$-quasi-Einstein manifolds $(M^n,g,X,\lambda)$. We establish
differential and integral identities for such manifolds. These identities
yield criteria for conformality, the Killing condition, and triviality,
and provide a unified framework for several rigidity phenomena. After this, we
recast the integrated Bochner formula as a Witten-type Hodge-energy
identity. The cancellation of its quartic term at $m=-2$ leads to a
triviality theorem in the generalized setting under a natural sign
condition on the integral of $\langle X,\nabla\lambda\rangle$.

An identity for the drift Laplacian gives a new proof of the previously known
triviality result for $m\leq-n$ and extends it to generalized
$m$-quasi-Einstein manifolds under a pointwise sign condition on
$\langle X,\nabla\lambda\rangle$. 
Finally, complementing our criterion characterizing when a conformal potential field is Killing, we exhibit in the appendix a closed generalized $m$-quasi-Einstein manifold whose potential field is conformal but non-Killing.

\end{abstract}

\maketitle
\tableofcontents
\section{Introduction}

An $n$-dimensional Riemannian manifold $(M^n,g)$, $n\ge 2,$ is called a \textit{generalized $m$-quasi-Einstein manifold} for a nonzero constant $m \in \mathbb{R}$, if there exist a smooth vector field $X$ on $M$ and a smooth function $\lambda: M \to \mathbb{R}$ satisfying the equation

\begin{equation}\label{GQE}
    \textrm{Ric}+\frac{1}{2} \mathcal{L}_Xg-\frac{1}{m}X^\flat \otimes X^\flat=\lambda g,
\end{equation}
where $\textrm{Ric}$, $\mathcal{L}_Xg$, $X^\flat$ and $\otimes$ stand for the Ricci tensor, the Lie derivative of the metric $g$ in the direction of $X$, the dual 1-form to $X$ associated with the metric $g$ and the tensor product, respectively. We denote a generalized $m$-quasi-Einstein manifold by $(M^n,g,X,\lambda)$. When $\lambda$ is constant, we shall simply refer to $(M^n,g,X,\lambda)$ as an $m$-quasi-Einstein manifold. We observe that, when $X$ vanishes identically, $M^n$ is an Einstein manifold. A generalized $m$-quasi-Einstein manifold will be called \textit{trivial} if the vector field X is identically zero. For comprehensive references on such a subject, we indicate \cite{BarrosGomes, barros2013compact, BarrosRibeiro2014}.

Typically, an $n$-dimensional $m$-quasi-Einstein manifold is precisely the manifold which is the base of an $(n+m)$-dimensional Einstein warped product, in the case $\lambda$ constant, $X$ a gradient vector field, and $m$ is a positive integer (see, for instance, \cite{case2011rigidity}). We point out that when $m$ goes to infinity and $\lambda$ is constant, the above equation reduces to a Ricci soliton, which generates self-similar solutions to the Ricci flow and often arises as singularity models. Also, formally allowing $m=\infty$ and $\lambda$ not constant, we have the notion of an almost Ricci soliton. For more information in this direction, see \cite{barros2012some} and references therein.

Another motivation for studying this equation, when $X$ is not necessarily gradient, is that the limiting procedure of taking extreme black hole spacetimes to near-horizon geometries gives a number of notable examples of $m$-quasi-Einstein manifolds that are also interesting from the viewpoint of geometric analysis and general relativity. In this scope, the reader can consult \cite{lucietti2026intrinsic, colling2026quasi}. In fact, the near-horizon equations in four-dimensional Einstein-Maxwell theory are a special case of the generalized 2-quasi-Einstein equations (see also \cite{kaminski2024extreme, colling2025rigidity}).  Furthermore, for other recent perspectives on this geometric structure, see \cite{cochran2026compact, valiyakath2026nilpotent, de2026generalized}.

We recall that a smooth vector field $Y$ defined on a Riemannian manifold $(M^n,g)$ is said to be \textit{conformal} if $\mathcal{L}_Y g=2\phi\, g$, for a smooth function $\phi$ on $M^n$. In particular, we see that $\phi$ satisfies $\mathrm{div}\,Y=n\phi,$ where $\mathrm{div}$ denotes the divergence of a vector field. In this context, $Y$ will be called \textit{Killing} if its conformal factor $\phi$ vanishes identically. We also say that a conformal vector field is nontrivial if it is not a Killing vector field. Conformal vector fields on a Riemannian manifold are an important tool for studying various analytical and geometric properties, and they have been studied extensively in the literature during the last decades. For more details, see the book \cite{sharma2024conformal}.

In \cite{BarrosGomes} Barros and Gomes proved that a closed $m$-quasi-Einstein manifold is trivial if $(M^n,g)$ is Einstein. Also, it is known (cf. \cite{BarrosGomes}) that if the potential vector field $X$ of a closed $m$-quasi-Einstein manifold is conformal, then it must be Killing.  More recently, Sharma \cite{Rsharma2025} also obtained the latter rigidity result by different methods.
Here we extend these results for closed generalized $m$-quasi-Einstein manifolds that satisfy a certain integral inequality. Namely, 

\begin{theo}\label{1}
Let $(M^n,g,X,\lambda)$, $n\geq2$, be a closed generalized
$m$-quasi-Einstein manifold whose potential field $X$ is conformal. Then
\begin{equation}\label{eq:C1}
\int_M\langle X,\nabla\lambda\rangle\,dM
=-\frac14\int_M|\mathcal L_Xg|^2\,dM
=-\frac1n\int_M(\operatorname{div}X)^2\,dM\leq0.
\end{equation}
Consequently, $X$ is Killing if and only if the first integral in
\eqref{eq:C1} vanishes; it is conformal non-Killing if and only if that
integral is strictly negative. If $n=2$, then necessarily $X=0$.
\end{theo}

The preceding theorem shows that, for conformal potential fields, the vanishing of \eqref{eq:C1} characterizes the Killing case, whereas strict negativity characterizes the non-Killing case. In the appendix, we show that the latter case actually occurs by constructing a closed generalized $m$-quasi-Einstein manifold admitting a conformal non-Killing potential vector field and satisfying
$
\int_M \langle X,\nabla\lambda\rangle\,dM<0.
$

We now turn to the Einstein setting, where a direct argument based solely on integral identities arising naturally from the generalized $m$-quasi-Einstein equation yields the following triviality criterion.

\begin{theo}[Barros--Gomes]\label{2}
Let $(M^n,g,X,\lambda)$, $n\ge 3$, be a closed Einstein generalized $m$-quasi-Einstein manifold.  Then
\begin{equation}\label{eq:C4}
\int_M\left((\operatorname{div}X)^2+
n\langle X,\nabla\lambda\rangle\right)dM
=-\frac{2(n-1)}{m^2(n+2)}\int_M|X|^4\,dM\leq0.
\end{equation}

The equality holds if and only if $X=0$.
\end{theo}

In addition, we obtain a new and more direct proof of the
integral formula in \cite[Proposition~2]{BarrosGomes}, based on the known
Pohozaev--Schoen identity. Consequently, our method provides an
alternative derivation of the Barros--Gomes rigidity result.

\begin{remark}
The argument above also recovers the rigidity result of
Barros--Gomes \cite[Remark~2]{BarrosGomes}. More precisely, if
$ 
\int_M\left((\operatorname{div}X)^2
+n\langle X,\nabla\lambda\rangle\right)\,dM\geq 0,
$
then our integral identity immediately yields $X= 0$. Note that the assumption
$
\int_M\langle X,\nabla\lambda\rangle\,dM\geq 0
$ 
implies the preceding condition.
\end{remark}

Moreover, under some assumptions, we establish the following  conformality result. Here, $S$ denotes the scalar curvature of $M$, defined as the trace of $\textrm{Ric}$.

\begin{theo}\label{3}
Let $(M^n,g,X,\lambda)$, $n\ge3,$ be a closed generalized $m$-quasi-Einstein manifold satisfying
$$
\int_M \langle X,\nabla S\rangle\,dM\le0
\quad\text{and}\quad
\frac1m\int_M X(|X|^2)\,dM\le0.
$$

Then $X$ is a conformal vector field. Moreover, $X$ is identically zero
if and only if $(M^n,g)$ is Einstein. Furthermore, if $$\int_M\langle X,\nabla\lambda\rangle\,dM<0,$$
then $X$ is not a Killing vector field.
\end{theo}
\begin{remark}
It is worth emphasizing that the integral expression
$$
(n-2)\int_M\langle X,\nabla S\rangle\,dM
+\frac{n+2}{m}\int_M X(|X|^2)\,dM
$$
is always nonnegative, see \eqref{eq:int-traceless-lie}, and it vanishes precisely
when $X$ is conformal. The assumptions of Theorem~\ref{3} force this
a priori nonnegative quantity to be nonpositive and hence to vanish,
which yields the conformality of $X$.
\end{remark}

We observe that  the traceless tensor associated with a tensor $T$ on $(M^n,g)$ is defined by $\mathring{T}=T-\frac{tr(T)}{n}g$, where $\operatorname{tr}(T)$ denotes the trace of $T$ with respect to $g$. In particular, the traceless Ricci tensor $\mathring{\operatorname{Ric}}$ is given by $\mathring{\operatorname{Ric}}=\operatorname{Ric}-\frac{S}{n}g$. With this information, we get the following result.

\begin{cor}\label{4}
There exists no closed generalized $m$-quasi-Einstein manifold
$(M^n,g,X,\lambda)$, $n\ge3$, with constant scalar curvature satisfying
$$
\frac1m\int_M X(|X|^2)\,dM\le0,\qquad
\int_M\langle X,\nabla\lambda\rangle\,dM<0,
$$
and
$$
\int_M
\mathring{\operatorname{Ric}}
(\nabla\operatorname{div}X,\nabla\operatorname{div}X)\,dM
\ge0.
$$
\end{cor}

\begin{remark}
Corollary~\ref{4} also clarifies the statement of Theorem~1.2 of Mirshafeazadeh and Bidabad \cite{ahmad2020rigidity}, where the authors assume the pointwise condition
$$
0<\frac{n}{2}X(|X|^2)\le |X|^2\operatorname{div}X
\quad\text{on }M.
$$
However, the strict inequality is incompatible with the closedness of $M$. Indeed, since $M$ is closed, the smooth function $|X|^2$ attains a maximum at some point $p_0\in M$. Hence,
$
\nabla(|X|^2)(p_0)=0,
$ 
and therefore
$ 
X(|X|^2)(p_0)=0,
$ 
contradicting the assumption that $X(|X|^2)>0$ everywhere. Furthermore, it is important to point out that the sphere rigidity conclusion of Theorem~1.2/Corollary 4.3 in \cite{ahmad2020rigidity} cannot be recovered simply by replacing the original hypothesis with
$ 
\frac{n}{2}X(|X|^2)\le |X|^2\operatorname{div}X,
$
or even by additionally assuming that $X$ is a conformal non-Killing vector
field (see Remark \ref{D} below). 
\end{remark}

Several rigidity criteria for the potential vector field of a closed $m$-quasi-Einstein manifold have appeared in the literature. For example, the equivalence between constant scalar curvature and the Killing condition was established by Ghosh~\cite[Theorem~4.2 and equation~(5.18)]{Gho20}. Under the restriction $m\neq-2$, Bahuaud et al.~\cite{Bahuaud2024} proved that $\operatorname{div}X=0$ implies that $X$ is Killing, while Cochran~\cite{Cochran2025} obtained the equivalence between the Killing condition and constant scalar curvature. Costa--Filho~\cite{CostaFilho2024} subsequently established both criteria for every admissible value of $m$, and Sharma~\cite[Theorem~1.1]{Rsharma2025} later obtained an integral criterion for the Killing condition. These results were extended to the generalized $m$-quasi-Einstein setting in \cite{de2026generalized}. In Section~\ref{sec:integral-identity-phi}, we provide a unified treatment of these results by deriving a single integral identity from which all the aforementioned criteria follow as direct consequences.

In Sections~\ref{sec:Witten} and~\ref{sec:proof-main}, we turn to a recent conjecture of Colling and Dunajski (see ~\cite[p.197]{colling2026quasi}), who expect Theorem~1.3 to remain valid for every $m<0$. More precisely, while their result establishes triviality for closed quasi-Einstein manifolds with $\lambda\leq0$ under the restriction $m\leq 2-n$, the authors hypothesize that the same conclusion should hold throughout the entire negative range. We state this as follows.

\begin{conjecture*}[Colling--Dunajski]
Every closed $m$-quasi-Einstein manifold with $\lambda\leq0$ and $m<0$ is trivial.
\end{conjecture*}

Our main result, Theorem~\ref{thm:main} below, settles this conjecture in the range $m\leq-2$. Before stating it, we record a collection of results obtained in Section~\ref{sec:Witten}, which include the specific values $m=-2$ and $m=-4$ in all dimensions, as well as related triviality results in the generalized setting. By means of several complementary identities and rigidity arguments developed in Section~\ref{sec:Witten}, we obtain a collection of results in this direction, which we summarize in the following theorem.

\begin{theo}\label{thm:intro-triviality}
Let $(M^n,g,X,\lambda)$, $n\geq2$, be a closed generalized
$m$-quasi-Einstein manifold with $\lambda\leq0$. Then $X\equiv0$ in
each of the following cases:
\begin{enumerate}
\item $m\leq-n$, provided
$    \langle X,\nabla\lambda\rangle\leq0
    \,\,\text{on }M
   $
when $n\geq3$;
\item $m=-2$, provided
$$
\int_M\langle X,\nabla\lambda\rangle\,dM\leq0;
$$
\item $m=-4$ and $\lambda$ is constant.
\end{enumerate}
\end{theo}

For constant $\lambda$, triviality was already known in the larger range $m\leq2-n$~\cite{colling2026quasi}. Part~(1) extends this rigidity phenomenon to generalized $m$-quasi-Einstein manifolds under a natural sign condition on the variation of $\lambda$ along $X$, suggesting a generalized counterpart of the Colling--Dunajski conjecture in the nonconstant-$\lambda$ setting. Part~(2) provides further evidence for this conjecture in generalized setting at the distinguished value $m=-2$ and, when $\lambda$ is constant, proves the original Colling--Dunajski statement at $m=-2$. Part~(3) proves the original conjecture at $m=-4$. In particular, for $n>7$, this value lies outside the previously known range $m\leq2-n$.

In the $m$-quasi-Einstein setting, parts~(2) and~(3) of
Theorem~\ref{thm:intro-triviality} are special instances of the main
result of this paper, which settles the Colling--Dunajski conjecture on
the whole half-line $(-\infty,-2]$. The theorem is:

\begin{theo}\label{thm:main}
Let $(M^n,g,X,\lambda)$, $n\geq2$, be a closed $m$-quasi-Einstein
manifold with constant $\lambda\leq0$. If $m\leq-2$, then
$X\equiv0$.
\end{theo}

For $n\geq5$, Theorem~\ref{thm:main} extends the range $m\leq2-n$ of
\cite[Theorem~1.3]{colling2026quasi} to $m\leq-2$, thereby covering the
additional interval $(2-n,-2]$. 

The proof of Theorem~\ref{thm:main} was obtained with the assistance of
artificial intelligence. Its key analytical ingredient,
Lemma~\ref{lem:power} below, was suggested by Astra~6.0, an OpenAI model
accessed through ChatGPT: for the operator $P_{a,b}=\delta_ad_b$,
built from the twisted exterior derivative $d_b$ and the twisted
codifferential $\delta_a$ introduced in Section~\ref{sec:twisted-energy},
testing a positive principal eigenfunction $\phi$ against the power
$\phi^{a/b}$ yields the weighted energy identity \eqref{eq:power}. With
$a=(m+2)/m$ and $b=-2/m$, this identity shows that the principal
eigenvalue is strictly positive whenever $m<-2$ and $X^\flat$ is not
exact. Combined with the differential identities of
Lemma~\ref{lem:intrinsic-codifferential} and
Proposition~\ref{prop:mixed-twisted-factorization}, this extends the
spectral and maximum-principle argument used for $m=-4$ in
Theorem~\ref{thm:rigidity-m-minus-four} to every $m<-2$, while the
endpoint $m=-2$ is covered by Corollary~\ref{cor:m-minus-two}. The proof
is given in Section~\ref{sec:proof-main}, and the AI contribution is
detailed in the AI Disclosure at the end of the paper.

\section{Conformal potential fields}

In what follows, all Riemannian manifolds $(M^n,g)$ are assumed to be connected
and oriented. When $M$ is closed, the divergence theorem yields
\begin{equation}\label{divergence}
\int_M\langle X,\nabla f\rangle dM
=-\int_M f\operatorname{div}X dM
\end{equation}
for every $f\in C^\infty(M)$. Since this identity will be used repeatedly throughout the text, we record it here for future reference.
Throughout the manuscript, we will use an important identity due to Barros and Ribeiro in \cite[p.~215]{BarrosRibeiro2012}. Its extension to the generalized gradient setting was obtained in \cite{BarrosRibeiro2014}. For the general case, this result was also proved by Ghosh in \cite[formula 6.5]{ghosh}. Here, $\Delta$ denotes the Laplacian operator on $M^n$.

\begin{lemma}\label{lem:laplacianX}
Let $(M^n,g,X,\lambda)$ be a generalized $m$-quasi-Einstein manifold. Then the following
identity holds
$$
\frac{1}{2}\,\Delta |X|^{2}
= |\nabla X|^{2}
- \mathrm{Ric}(X,X)
+ \frac{2}{m}\,|X|^{2}\,\mathrm{div}\,X
+ (2-n)\,\langle X,\nabla \lambda\rangle.
$$
\end{lemma}

Before proceeding to the proof of Theorem~\ref{1}, we record another identity that will be used repeatedly. Taking the trace of \eqref{GQE}, we obtain
\begin{equation}\label{eq2}
S+\operatorname{div}X=\frac{1}{m}|X|^2+n\lambda.
\end{equation}

As an immediate consequence, we make the following observation: if $M^n$ is closed, integrating \eqref{eq2} and applying the divergence theorem yields the following.
$$
\int_M (S-n\lambda)\,dM=\frac{1}{m}\int_M |X|^2\,dM.
$$

Hence, under the condition $m(S-n\lambda)\leq0$, we conclude that necessarily $X\equiv0$ (see the analogous fact in \cite{BarrosRibeiro2014}, Theorem 2, item 2).

We are now ready to prove our first result.

\begin{proof}[\textbf{Proof of Theorem} \ref{1}]
We begin with the following integral formula for a closed Riemannian manifold (cf. \cite{Yan70})
\begin{equation}\label{eq:bochner}
\frac12\int_M |\mathcal{L}_Xg|^2\,dM
=
\int_M
\bigl(
|\nabla X|^2
-
\operatorname{Ric}(X,X)
\bigr)\,dM
+
\int_M (\operatorname{div}X)^2\,dM.
\end{equation}

Now, integrating the identity of Lemma \ref{lem:laplacianX} on the closed manifold $M$ and applying the divergence theorem, we conclude that,
\begin{equation}\label{eq:key}
\int_M
\bigl(
|\nabla X|^2
-
\operatorname{Ric}(X,X)
\bigr)\,dM
=
(n-2)
\int_M
\langle X,\nabla\lambda\rangle\,dM
-
\frac{2}{m}
\int_M
|X|^2\operatorname{div}X\,dM.
\end{equation}

Assume now that $X$ is conformal. On the one hand, we have $\int_M
|X|^2\operatorname{div}X\,dM=0$ (cf. \cite[Lemma 1, item (2)]{BarrosRibeiro2014}).
Hence, substituting \eqref{eq:key} into \eqref{eq:bochner}, we obtain
\begin{equation}\label{eq:conformal}
\frac12\int_M |\mathcal{L}_Xg|^2\,dM
=
(n-2)
\int_M
\langle X,\nabla\lambda\rangle\,dM
+
\int_M(\operatorname{div}X)^2\,dM.
\end{equation}

On the other hand, we recall the known Kazdan–Warner identity $\int_M \langle X,\nabla S\rangle\,dM=0$ (see  Bourguignon and Ezin \cite[Theorem II.9]{bourguignon1987scalar}). In view of \eqref{eq2}, we have
$$
\int_M
\langle X,\nabla(\operatorname{div}X)\rangle\,dM
=
n
\int_M
\langle X,\nabla\lambda\rangle\,dM
+
\frac1m
\int_M
\langle X,\nabla|X|^2\rangle\,dM.
$$

Applying \eqref{divergence} with $f=\operatorname{div}X$ and $f=|X|^2$, respectively, gives

$$
\int_M(\operatorname{div}X)^2\,dM
=
-\int_M\langle X,\nabla\operatorname{div}X\rangle\,dM,
$$
and
$$
\int_M\langle X,\nabla|X|^2\rangle\,dM
=
-\int_M|X|^2\operatorname{div}X\,dM.
$$

Substituting these identities into \eqref{eq:conformal}, we obtain
\begin{equation}\label{eq:mainidentity}
\frac12\int_M |\mathcal{L}_Xg|^2\,dM
=
-2
\int_M
\langle X,\nabla\lambda\rangle\,dM.
\end{equation}

Since $|\mathcal L_Xg|^2=4(\operatorname{div}X)^2/n$, this is precisely
\eqref{eq:C1}, including both equality characterizations.

If $n=2$, then $\mathring{\operatorname{Ric}}=0$ identically.  Taking the
trace-free part of \eqref{GQE} and using
$\mathring{\mathcal{L}_Xg}=0$ gives
$
 \mathring{(X^\flat\otimes X^\flat)}=0.
$
If $X\neq0$, evaluation on a unit vector orthogonal to $X$
gives $-|X|^2/2=0$, a contradiction. Hence $X\equiv0$.

This concludes the proof of Theorem \ref{1}.
\end{proof}

\begin{remark}
In the gradient setting, Barros and Ribeiro~\cite[Theorem 3, item 2]{BarrosRibeiro2014} proved that every Einstein generalized $m$-quasi-Einstein gradient manifold with conformal potential vector field is trivial. The same conclusion holds immediately in the non-gradient setting. Indeed, if $M$ is Einstein and $X$ is conformal, then taking the traceless part of the generalized $m$-quasi-Einstein equation \eqref{GQE} yields $ -\frac{1}{m}\,\mathring{(X^\flat\otimes X^\flat)}=0,$ which forces $X\equiv0$.    
\end{remark}

\section{Pohozaev--Schoen identities and applications}
Having treated the conformal case, we now derive integral identities
without imposing any a priori condition on the potential field. Our
starting point is the Pohozaev--Schoen identity, which expresses the
average variation of the scalar curvature along $X$ in terms of the
$L^2$-pairing between $\mathring{\operatorname{Ric}}$ and
$\mathcal L_Xg$. Combining this identity with the
trace-free part of the generalized $m$-quasi-Einstein equation yields
the formulas developed in this section.

\subsection{The basic integral identities}

The \textit{Pohozaev--Schoen identity} is a powerful geometric integral formula for analyzing different problems on closed manifolds, with numerous applications in partial differential equations and mathematical physics. We refer the reader to Barbosa, Freitas, and de Lima \cite{barbosa} for further discussion and applications.

\begin{lemma}\label{PS} Let $(M^n,g)$, $n\ge3,$ be a closed Riemannian manifold and $Y$ a smooth vector field on $M^n.$ Then $$\int_M\langle Y,\nabla S\rangle \,dM=-\frac{n}{n-2}\int_M\langle \mathring{\operatorname{Ric}}, \mathcal{L}_Yg\rangle \,dM.$$
\end{lemma}

 Here, for a tensor $T$ on a Riemannian manifold $(M^n,g)$, we consider $|\mathring{T}|^2=|T|^2-\frac{tr(T)^2}{n}$ the Hilbert-Schmidt norm of $\mathring{T}.$ In this scenario, we have the following result,

\begin{proposition}\label{prop:integral-identities}
Let $(M^n,g,X,\lambda),$ $n\ge3,$ be a closed generalized $m$-quasi-Einstein manifold. Then the following integral identities hold:
\begin{align}
\int_M |\mathring{\Ric}|^2\,dM
&=\frac{n-2}{2n}\int_M \langle X,\nabla S\rangle\,dM+
\frac{1}{m}\int_M \mathring{\Ric}(X,X)\,dM,
\label{eq:int-traceless-ricci}\\
\frac12\int_M |\mathring{\mathcal{L}_Xg}|^2\,dM
&= \frac{n-2}{n}
\int_M\langle X,\nabla S\rangle\,dM
+\frac{n+2}{mn}
\int_M  X(|X|^2)\,dM.
\label{eq:int-traceless-lie}
\end{align}
\end{proposition}

\begin{proof}
Since $M$ is closed, Lemma \ref{PS} yields
\begin{equation}\label{eq:PSclosed}
\int_M \langle X,\nabla S\rangle\,dM
=
-\frac{n}{n-2}
\int_M \langle \mathring{\Ric}, \mathcal{L}_Xg\rangle\,dM.
\end{equation}

Taking the traceless part of \eqref{GQE}, we obtain
$$
\mathring{\Ric}
+\frac12 \mathring{\mathcal L_Xg}
-\frac1m \mathring{(X^\flat\otimes X^\flat)} 
=0,
$$ and therefore
$$ 
\mathring{\mathcal L_Xg}
=
-2\mathring{\Ric}
+\frac{2}{m}\mathring{(X^\flat\otimes X^\flat)}.
$$

Since $\mathring{\Ric}$ is trace-free,
$$
\bigl\langle \mathring{\Ric},
\mathring{(X^\flat\otimes X^\flat)}
\bigr\rangle
=
\langle \mathring{\Ric},
X^\flat\otimes X^\flat
\rangle
=
\mathring{\Ric}(X,X).
$$

Hence,
\begin{equation}\label{eq:key-inner}
\langle \mathring{\Ric}, \mathcal{L}_Xg\rangle
=
-2|\mathring{\Ric}|^2
+\frac{2}{m}\mathring{\Ric}(X,X).
\end{equation}

Substituting \eqref{eq:key-inner} into \eqref{eq:PSclosed}, we obtain
$$
\int_M \langle X,\nabla S\rangle\,dM
=
\frac{2n}{n-2}
\int_M |\mathring{\Ric}|^2\,dM
-\frac{2n}{m(n-2)}
\int_M \mathring{\Ric}(X,X)\,dM,
$$
which proves \eqref{eq:int-traceless-ricci}. On the other hand, by using again that
$$
\mathring{\mathcal L_Xg}
=
-2\mathring{\Ric}
+\frac{2}{m}\mathring{(X^\flat\otimes X^\flat)},
$$
we compute
$$
\frac12|\mathring{\mathcal L_Xg}|^2
=
-\langle\mathring{\Ric},\mathcal{L}_Xg\rangle
+\frac1m
\left(
(\mathcal{L}_Xg)(X,X)
-\frac{2}{n}|X|^2\operatorname{div}X
\right).
$$

Integrating this equality and applying \eqref{eq:PSclosed}, we obtain
$$
\frac12
\int_M |\mathring{\mathcal L_Xg}|^2\,dM
=
\frac{n-2}{n}
\int_M\langle X,\nabla S\rangle\,dM
+\frac1m
\int_M
\left(
X(|X|^2)
-\frac{2}{n}|X|^2\operatorname{div}X
\right)dM.
$$

As before, applying \eqref{divergence} with $f=|X|^2$ gives

$$
\int_M X(|X|^2)\,dM
=
-\int_M |X|^2\operatorname{div}X\,dM.
$$

Hence,
$$
\frac12
\int_M |\mathring{\mathcal L_Xg}|^2\,dM
=
\frac{n-2}{n}
\int_M\langle X,\nabla S\rangle\,dM
+\frac{n+2}{mn}
\int_M  X(|X|^2)\,dM,
$$
which proves \eqref{eq:int-traceless-lie}.
\end{proof}

For the reader's convenience, we include another proof of \eqref{eq:int-traceless-lie}, based on Lemma~\ref{lem:laplacianX}.

\begin{proof}[\textbf{Alternative proof of ~\eqref{eq:int-traceless-lie}}]
Since
$$
\mathring{\mathcal L_Xg}
=
\mathcal L_Xg-\frac{2}{n}(\operatorname{div}X)g,
$$
we have
$$
|\mathring{\mathcal L_Xg}|^2
=
|\mathcal L_Xg|^2
-\frac{4}{n}(\operatorname{div}X)^2.
$$

Hence,
$$
\frac12
\int_M
|\mathring{\mathcal L_Xg}|^2\,dM
=
\frac12
\int_M
|\mathcal L_Xg|^2\,dM
-\frac{2}{n}
\int_M
(\operatorname{div}X)^2\,dM.
$$

We remember the Yano's integral formula (\ref{eq:bochner})
$$
\frac12
\int_M
|\mathcal L_Xg|^2\,dM
=
\int_M
\bigl(|\nabla X|^2-\Ric(X,X)\bigr)\,dM
+
\int_M
(\operatorname{div}X)^2\,dM.
$$

Moreover, Lemma \ref{lem:laplacianX} gives
$$
\int_M
\bigl(|\nabla X|^2-\Ric(X,X)\bigr)\,dM
=
(n-2)\int_M\langle X,\nabla\lambda\rangle\,dM
-\frac{2}{m}
\int_M
|X|^2\operatorname{div}X\,dM.
$$
Therefore,

$$
\frac12
\int_M
|\mathring{\mathcal L_Xg}|^2\,dM
=
(n-2)\int_M\langle X,\nabla\lambda\rangle\,dM
+\frac{n-2}{n}
\int_M
(\operatorname{div}X)^2\,dM
-\frac{2}{m}
\int_M
|X|^2\operatorname{div}X\,dM.
$$

To eliminate the derivative of $\lambda$, we differentiate the trace identity \eqref{eq2} along $X$, obtaining

$$
n\langle X,\nabla\lambda\rangle
=
\langle X,\nabla S\rangle
-\frac{1}{m}X(|X|^2)
+\langle X,\nabla(\operatorname{div}X)\rangle.
$$

Applying \eqref{divergence} to
$f=\operatorname{div}X$ and $f=|X|^2$, and substituting into the preceding identity, we obtain, after cancellation,

$$
\frac12
\int_M
|\mathring{\mathcal L_Xg}|^2\,dM
=
\frac{n-2}{n}
\int_M
\langle X,\nabla S\rangle\,dM
+
\frac{n+2}{mn}
\int_M
X(|X|^2)\,dM,
$$

which proves \eqref{eq:int-traceless-lie}.

\end{proof}
As an immediate consequence, we obtain the next corollary.

\begin{corollary}
Under the assumptions of Proposition \ref{prop:integral-identities}, if the scalar curvature is constant, then
$$
\int_M |\mathring{\Ric}|^2\,dM
=
\frac{1}{m}
\int_M \mathring{\Ric}(X,X)\,dM.
$$

In particular, if $\frac{1}{m}
\mathring{\Ric}(X,X)\leq 0$ on $M$, then $(M^n,g)$ is Einstein.
\end{corollary}

Another application of Proposition \ref{prop:integral-identities} is given by the following result.
        
\begin{proposition}\label{prop:traceless-ricci}
Let $(M^n,g,X,\lambda)$ be a closed generalized $m$-quasi-Einstein manifold. Then
\begin{align}
\int_M |\mathring{\Ric}|^2\,dM
&=
\frac{n-2}{2n}\int_M \langle X,\nabla S\rangle\,dM \notag\\
&\quad
+\frac{n+2}{2mn}
\int_M |X|^2\operatorname{div}X\,dM
+\frac{n-1}{m^2n}
\int_M |X|^4\,dM.
\label{eq:main-integral}
\end{align}

\end{proposition}

\begin{proof}
First, a straightforward calculation shows that
$$\mathring{\Ric}(X,X)
=
-\langle \nabla_XX,X\rangle
+\frac{|X|^2}{n}\operatorname{div}X
+\frac{n-1}{mn}|X|^4.
$$
Indeed, taking the traceless part of the generalized $m$-quasi-Einstein equation yields
$$
\mathring{\Ric}
=
-\frac12( \mathcal{L}_Xg)^\circ
+\frac1m(X^\flat\otimes X^\flat)^\circ.
$$
Evaluating at $(X,X)$ gives
$$
\mathring{\Ric}(X,X)
=
-\langle \nabla_XX,X\rangle
+\frac{|X|^2}{n}\operatorname{div}X
+\frac{n-1}{mn}|X|^4,
$$
which proves the claimed identity.

Since
$$
\langle \nabla_XX,X\rangle
=
\frac12X(|X|^2),
$$
and $M$ is closed, \eqref{divergence} applied once more with $f=|X|^2$, yields

$$
\int_M X(|X|^2)\,dM
=
-\int_M |X|^2\operatorname{div}X\,dM.
$$

Therefore,
$$
\int_M \mathring{\Ric}(X,X)\,dM
=
\frac{n+2}{2n}
\int_M |X|^2\operatorname{div}X\,dM
+\frac{n-1}{mn}
\int_M |X|^4\,dM.
$$

Substituting this identity into \eqref{eq:int-traceless-ricci}
yields \eqref{eq:main-integral}. 
\end{proof}

As a corollary, we can deduce the following integral equality, originally obtained in \cite{BarrosGomes}.

\begin{corollary}\label{cor:integral-identities}
Let $(M^n,g,X,\lambda)$ be a closed generalized $m$-quasi-Einstein manifold. Then
\begin{align*}
\int_M |\mathring{\Ric}|^2\,dM
&=
-\frac{2}{n}\int_M\langle X,\nabla S\rangle\,dM  \\
&\quad
+\frac{n+2}{2n}
\int_M\Big((\operatorname{div}X)^2
+n\langle X,\nabla\lambda\rangle\Big)\,dM
+\frac{n-1}{m^2n}
\int_M|X|^4\,dM.
\end{align*}
\end{corollary}

\begin{proof}
From the traced generalized $m$-quasi-Einstein equation \eqref{eq2} we obtain
$$
\frac1m|X|^2\operatorname{div}X
=(\operatorname{div}X)^2
+(S-n\lambda)\operatorname{div}X.
$$

Integrating over $M$ and applying \eqref{divergence} with $f=S$ and $f=\lambda$, we obtain
$$
\frac1m\int_M |X|^2\operatorname{div}X\,dM
=
\int_M(\operatorname{div}X)^2\,dM
-\int_M\langle X,\nabla S\rangle\,dM
+n\int_M\langle X,\nabla\lambda\rangle\,dM.
$$

Substituting this identity into Proposition~\ref{prop:traceless-ricci}, we obtain
$$
\begin{aligned}
\int_M |\mathring{\Ric}|^2\,dM
&=
\frac{n-2}{2n}\int_M\langle X,\nabla S\rangle\,dM  \\
&\quad
+\frac{n+2}{2n}
\left(
\int_M(\operatorname{div}X)^2\,dM
-\int_M\langle X,\nabla S\rangle\,dM
+n\int_M\langle X,\nabla\lambda\rangle\,dM
\right) \\
&\quad
+\frac{n-1}{m^2n}\int_M|X|^4\,dM.
\end{aligned}
$$

Thus, the desired equality follows.
\end{proof}

\begin{remark}
Corollary~\ref{cor:integral-identities} provides an alternative derivation of the integral identity established by Barros and Gomes in Proposition~2 of \cite{BarrosGomes}. In contrast to the original proof, which relies on the Hodge--de Rham decomposition and several auxiliary integral identities, our argument follows directly from the Pohozaev--Schoen identity and elementary tensorial computations. Thus, the novelty lies not in the identity itself but in its considerably shorter and more direct proof.
\end{remark}

\subsection{The Einstein case}
We first apply the preceding identities to the Einstein setting.
Since the scalar curvature is then constant, the Corollary~\ref{cor:integral-identities}
reduces to an identity involving only $\operatorname{div}X$,
$\langle X,\nabla\lambda\rangle$, and $|X|^4$. This immediately
yields Theorem~\ref{2}.
\begin{proof}[\textbf{Proof of Theorem~\ref{2}}]
Since $(M^n,g)$ is Einstein and $n\geq 3$, its scalar curvature is
constant. Hence, Corollary~\ref{cor:integral-identities} yields
$$
\frac{n+2}{2n}
\int_M\left((\operatorname{div}X)^2
+n\langle X,\nabla\lambda\rangle\right)\,dM
+\frac{n-1}{m^2n}\int_M |X|^4\,dM=0.
$$

Multiplying this identity by $\frac{2n}{n+2}$, we obtain
$$
\int_M\left((\operatorname{div}X)^2
+n\langle X,\nabla\lambda\rangle\right)\,dM
=
-\frac{2(n-1)}{m^2(n+2)}
\int_M|X|^4\,dM\le0.
$$

This proves \eqref{eq:C4}. Moreover, equality holds if and only if
$
\int_M|X|^4\,dM=0,
$
which is equivalent to $X\equiv0$.
\end{proof}

\subsection{A conformality criterion}
Next, we use the second identity in Proposition~\ref{prop:integral-identities} to detect
conformality. Its left-hand side is the $L^2$-norm of the trace-free
part of $\mathcal L_Xg$. Consequently, suitable sign assumptions on
the right-hand side force the conformal defect of $X$ to vanish.
\begin{proof}[\textbf{Proof of Theorem} \ref{3}]
We first show that $X$ is a conformal vector field. 
By \eqref{eq:int-traceless-lie},
$$
\frac12\int_M|\mathring{\mathcal L_Xg}|^2\,dM
=
\frac{n-2}{n}\int_M\langle X,\nabla S\rangle\,dM
+
\frac{n+2}{mn}\int_MX(|X|^2)\,dM.
$$

On the other hand, the assumptions imply that the right-hand side is nonpositive. Hence,
$ 
\mathring{\mathcal L_Xg}=0,
$
that is, $X$ is conformal.

If $X\equiv0$, then the generalized $m$-quasi-Einstein equation \eqref{GQE} immediately shows that $(M,g)$ is Einstein. Conversely, if $(M,g)$ is Einstein, then $S$ is constant. Therefore, \eqref{eq:int-traceless-lie} gives
$
\int_MX(|X|^2)\,dM=0,
$
and the divergence theorem yields
$
\int_M|X|^2\operatorname{div}X\,dM=0.
$
Substituting these identities into \eqref{eq:main-integral}, we obtain
$$
0=\frac{n-1}{m^2n}\int_M|X|^4\,dM,
$$
which implies $X\equiv0$.

Finally, since $X$ is conformal, identity \eqref{eq:mainidentity} yields
$$
\frac12\int_M|\mathcal{L}_Xg|^2\,dM
=
-2\int_M\langle X,\nabla\lambda\rangle\,dM.
$$

If
$
\int_M\langle X,\nabla\lambda\rangle\,dM<0,
$
then the right-hand side is positive, so $\mathcal{L}_Xg\neq 0$. Therefore, $X$ is not a Killing vector field.
\end{proof}

\begin{remark}\label{rem:comparison_theoremA}
Assume throughout that
$ 
\int_M \langle X,\nabla S\rangle\,dM\le0.
$ 
Under this hypothesis, Theorem~\ref{3} complements Theorem~A due to de Carvalho, Lima, and Costa-Filho in \cite{de2026generalized}. Taken together, these results provide a description of the potential vector field $X$, according to the sign of the integral
$
\int_M \langle X,\nabla\lambda\rangle\,dM.
$
More precisely, we infer that
\begin{enumerate}
    \item[\rm (i)]If
    $
    \int_M \langle X,\nabla\lambda\rangle\,dM\ge0,
    $
    then Theorem~A in \cite{de2026generalized} shows that $X$ is necessarily a Killing vector field.

    \item[\rm (ii)]If
    $$
    \int_M \langle X,\nabla\lambda\rangle\,dM<0
    \quad\text{and}\quad
    \frac1m\int_M X(|X|^2)\,dM\le0,
    $$
    then Theorem~\ref{3} implies that $X$ is a conformal vector field which is not Killing.
\end{enumerate}

Therefore, under the above assumptions, the sign of
$
\int_M \langle X,\nabla\lambda\rangle\,dM
$
completely determines whether the potential vector field is Killing or strictly conformal.

Moreover, in contrast to Theorem~A in \cite{de2026generalized}, our approach based on the integral identities~\eqref{eq:int-traceless-ricci}--~\eqref{eq:int-traceless-lie} yields a complete characterization of the trivial case: a generalized m-quasi-Einstein structure is trivial if and only if the underlying manifold $(M^n,g)$ is Einstein.
\end{remark}

In the sequel, we deduce the last result in this section.

\begin{proof}[\textbf{Proof of Corollary} \ref{4}]
Suppose, for contradiction, that there exists a closed generalized
$m$-quasi-Einstein manifold $(M^n,g,X,\lambda)$ satisfying the hypotheses
of the corollary.

Since the scalar curvature $S$ is constant, we have
$ 
\nabla S=0,
$
and hence,
$
\int_M \langle X,\nabla S\rangle\,dM=0.
$
Therefore, all the assumptions of Theorem~\ref{3} are fulfilled. It follows that $X$ is a conformal vector field.  Furthermore, since $\int_M \langle X, \nabla \lambda \rangle \, dM < 0$, $X$ is a non-Killing vector field (that is, $\operatorname{div} X \not= 0$).

By a classical characterization theorem due to Yano (cf. \cite[p. 56]{Yan70}), any closed Riemannian manifold $(M^n, g)$, $n>2,$ with constant scalar curvature that admits a non-Killing conformal vector field $X$ satisfying $\int_M \mathring{Ric}(\nabla \operatorname{div} X, \nabla \operatorname{div} X) \, dM \ge 0$ is globally isometric to a standard round sphere. In particular, $(M^n,g)$ is an Einstein manifold. Therefore Theorem~\ref{3} yields that $X\equiv0$, a contradiction with the assumption that the generalized $m$-quasi-Einstein structure is nontrivial. Therefore, no such manifold exists.
\end{proof}

\begin{remark}\label{D}
As mentioned previously, the sphere-rigidity conclusion in \cite{ahmad2020rigidity} cannot be recovered by any of these modifications of the
hypotheses. Indeed, the original strict inequality is used to guarantee that $X$ is
non-Killing, whereas the non-strict inequality, together with the assumption
that the scalar curvature is constant, implies that $X$ is conformal. However, every conformal vector field on an Einstein generalized
$m$-quasi-Einstein manifold is necessarily trivial. Hence $X\equiv0$, which
contradicts that $X$ is non-Killing.     
\end{remark}

\section{Characterizations of Killing potential fields}
\label{sec:integral-identity-phi}

Let $(M^n,g,X,\lambda)$ be a closed generalized
$m$-quasi-Einstein manifold. To study some properties of the potential vector
field $X$, we can combine the divergence of equation \eqref{GQE} with its trace
\eqref{eq2}. This procedure naturally singles out the scalar function
$
\Phi=2S-(n+2)\lambda.
$
In the next proposition, the resulting integral identity expresses the variation of $\Phi$ along the
flow of $X$ as a sum of nonnegative terms involving $\mathcal L_Xg$ and
$\operatorname{div}X$. It therefore provides a direct criterion for the
Killing condition, recovers Theorem~A and Proposition~B of
\cite{de2026generalized}, and yields further rigidity consequences in both the
generalized and classical settings.

We begin with the divergence of the generalized $m$-quasi-Einstein equation.
The contracted Bianchi identity and the formula
\begin{equation}\label{div XX}
 \operatorname{div}(X^\flat\otimes X^\flat)
 =(\operatorname{div}X)X^\flat+(\nabla_XX)^\flat,
\end{equation}
give, after taking the divergence of \eqref{GQE},
\begin{equation}\label{eq:divergence-section2}
 \frac12\,\nabla S
 +\frac12\bigl(\operatorname{div}(\mathcal L_Xg)\bigr)^\sharp
 -\frac1m\bigl((\operatorname{div}X)X+\nabla_XX\bigr)
 =\nabla\lambda.
\end{equation}

\begin{remark}\label{rem:killing-gradient-relations}
This argument is pointwise and requires no compactness assumption. Thus, if $X$
is Killing, then \eqref{eq:divergence-section2} and the gradient of the
traced equation give, respectively,
$\nabla S-\frac{2}{m}\nabla_XX=2\nabla\lambda$ and
$\nabla S+\frac{2}{m}\nabla_XX=n\nabla\lambda$, where we used
$\nabla|X|^2=-2\nabla_XX$. Hence
$\nabla S=\frac{n+2}{2}\nabla\lambda$, and so
$\Phi=2S-(n+2)\lambda$ is constant on $M$.
In particular, if $\lambda$ is constant, then so is $S$.
\end{remark}
We are now ready to establish the following result.

\begin{proposition}\label{prop:phi-integral}
Let $(M^n,g,X,\lambda)$ be a closed generalized $m$-quasi-Einstein
manifold. Define the function 
$
\Phi=2S-(n+2)\lambda
$ on $M^n$.
Then
\begin{equation}\label{eq:phi-integral}
 \int_M\langle X,\nabla\Phi\rangle\,dM
 =\frac12\int_M|\mathcal L_Xg|^2\,dM
 +\int_M(\operatorname{div}X)^2\,dM.
\end{equation}
\end{proposition}

\begin{proof}
Since $\mathcal L_Xg$ is symmetric, integration by parts gives
$$
 \int_M\langle\operatorname{div}(\mathcal L_Xg),X\rangle\,dM
 =-\frac12\int_M|\mathcal L_Xg|^2\,dM.
$$

A further application of \eqref{divergence}, with $f=S$, $f=\lambda$ and $f=|X|^2$ respectively, gives
$$
\int_M\langle\nabla S,X\rangle\,dM
=
-\int_M S\,\operatorname{div}X\,dM,
\qquad
\int_M\langle\nabla\lambda,X\rangle\,dM
=
-\int_M \lambda\,\operatorname{div}X\,dM.
$$
and
$$
 \begin{aligned}
 &\int_M\left(|X|^2\operatorname{div}X
 +\langle\nabla_XX,X\rangle\right)dM \\
 &\qquad
 =\int_M\left(|X|^2\operatorname{div}X
 +\frac12X(|X|^2)\right)dM
 =\frac12\int_M|X|^2\operatorname{div}X\,dM.
 \end{aligned}
$$

Contracting \eqref{eq:divergence-section2} with $X$, integrating, and using
the preceding identities, we find
$$
 \frac14\int_M|\mathcal L_Xg|^2\,dM
 =\int_M\left(\lambda-\frac S2-\frac{|X|^2}{2m}\right)
 \operatorname{div}X\,dM.
$$

By \eqref{eq2},
$$
 \lambda-\frac S2-\frac{|X|^2}{2m}
 =-\frac12\Phi-\frac12\operatorname{div}X.
$$

Consequently,
$$
 \frac12\int_M|\mathcal L_Xg|^2\,dM
 =-\int_M\Phi\operatorname{div}X\,dM
 -\int_M(\operatorname{div}X)^2\,dM.
$$

Finally, applying \eqref{divergence} with $f=\Phi$, we obtain
$$
-\int_M\Phi\,\operatorname{div}X\,dM
=
\int_M\langle X,\nabla\Phi\rangle\,dM,
$$
which proves \eqref{eq:phi-integral}.

\end{proof}

\begin{remark}\label{rem:phi-previous-arxiv}
The identity \eqref{eq:phi-integral} had already been obtained, in equivalent 
form, in the proof of Theorem~A of
\cite{de2026generalized}. More precisely, it is the unnumbered
equation displayed immediately below equation~(6) on page~7 of
\cite{de2026generalized}:
$$
 \frac12\int_M|\mathcal L_Xg|^2\,dM
 =-\int_M\left((\operatorname{div}X)^2
 +(n+2)\langle X,\nabla\lambda\rangle
 -2\langle X,\nabla S\rangle\right)dM.
$$

Rearranging the terms and using
$\nabla\Phi=2\nabla S-(n+2)\nabla\lambda$ gives precisely
\eqref{eq:phi-integral}. Thus, the above Proposition is not a new integral
identity. We include it to emphasize the geometric role of $\Phi$ and the
rigidity consequence needed in the following sections.
\end{remark}

\begin{remark}\label{rem:phi-from-integral-identities}
When $n\geq3$, identity \eqref{eq:phi-integral} can also be deduced from
identity~\eqref{eq:int-traceless-lie}, namely:
\begin{equation}\label{eq:traceless-lie-identity}
 \frac12\int_M| \mathring{\mathcal L_Xg}|^2 \,dM
 =\frac{n-2}{n}\int_M\langle X,\nabla S\rangle\,dM
 +\frac{n+2}{mn}\int_M X(|X|^2)\,dM.
\end{equation}

Indeed, differentiating \eqref{eq2} in the direction of $X$, integrating, and applying
the divergence theorem, we obtain
$$
 \frac1m\int_MX(|X|^2)\,dM
 =\int_M\langle X,\nabla S\rangle\,dM
 -\int_M(\operatorname{div}X)^2\,dM
 -n\int_M\langle X,\nabla\lambda\rangle\,dM.
$$

Furthermore,
$$
 \frac12\int_M| \mathring{\mathcal L_Xg}|^2 \,dM
 =\frac12\int_M|\mathcal L_Xg|^2\,dM
 -\frac2n\int_M(\operatorname{div}X)^2\,dM.
$$

Substituting these two expressions into
\eqref{eq:traceless-lie-identity}, we arrive at
$$
 \frac12\int_M|\mathcal L_Xg|^2\,dM
 =2\int_M\langle X,\nabla S\rangle\,dM
 -(n+2)\int_M\langle X,\nabla\lambda\rangle\,dM
 -\int_M(\operatorname{div}X)^2\,dM,
$$
which is exactly \eqref{eq:phi-integral}.
\end{remark}

Proposition \ref{prop:phi-integral} allows us to deduce the following fact.

\begin{corollary}\label{cor:phi-killing}
Let $(M^n,g,X,\lambda)$ be a closed generalized $m$-quasi-Einstein manifold. If
$$
\int_M \langle X,\nabla\Phi\rangle\,dM\leq 0,
$$
then $X$ is a Killing vector field.
\end{corollary}

\begin{proof}
By hypothesis, the left-hand side of \eqref{eq:phi-integral} is nonpositive,
while each term on the right-hand side is nonnegative. Hence, both sides
must vanish, and in particular, $X$ is a Killing vector field.
\end{proof}

\begin{remark}
Although the function $\Phi$ was not identified in
\cite{de2026generalized}, the corollary above already follows from the
unnumbered identity recalled in Remark~\ref{rem:phi-previous-arxiv} after
setting $\Phi=2S-(n+2)\lambda$. In that paper, the result was instead
formulated in terms of separate integral inequalities involving $S$ and
$\lambda$.
\end{remark}

Moreover, Theorem~A of \cite{de2026generalized} is an immediate consequence of
Proposition~\ref{prop:phi-integral}.

\begin{corollary}[Theorem~A of \cite{de2026generalized}]
\label{cor:theorem-a-section2}
Let $(M^n,g,X,\lambda)$ be a closed generalized $m$-quasi-Einstein manifold. If
$$
\int_M\langle X,\nabla S\rangle\,dM\leq0
\quad\text{and}\quad
\int_M\langle X,\nabla\lambda\rangle\,dM\geq0,
$$
then $X$ is a Killing vector field.
\end{corollary}

\begin{proof}
The assumptions imply
$$
 \int_M\langle X,\nabla\Phi\rangle\,dM
 =2\int_M\langle X,\nabla S\rangle\,dM
 -(n+2)\int_M\langle X,\nabla\lambda\rangle\,dM\leq0.
$$

The conclusion now follows directly from Corollary~\ref{cor:phi-killing}.
\end{proof}

The same identity \eqref{eq:phi-integral} also directly gives the following corollary.

\begin{corollary}[Proposition~B of \cite{de2026generalized}]
\label{cor:proposition-b-section2}
Let $(M^n,g,X,\lambda)$ be a closed generalized $m$-quasi-Einstein manifold. If $\operatorname{div}X=0$, then $X$ is a Killing vector field.
\end{corollary}

\begin{proof}
Applying \eqref{divergence} with $f=\Phi$, we obtain
$$
\int_M\langle X,\nabla\Phi\rangle\,dM
=
-\int_M\Phi\,\operatorname{div}X\,dM
=
0.
$$

It now follows from \eqref{eq:phi-integral} that $\mathcal L_Xg=0$.
\end{proof}

Next, we recover the following result.

\begin{corollary}[Theorem~A and Theorem~C in \cite{de2026generalized}]
\label{cor:theorem-c-section2}
Let $(M^n,g,X,\lambda)$ be a closed generalized $m$-quasi-Einstein manifold.
Then the following conditions are equivalent:
\begin{enumerate}
    \item[\rm (i)]
    $ 
    \int_M \langle X,\nabla\Phi\rangle\,dM\leq 0;
    $
    \item[\rm (ii)] $X$ is a Killing vector field;
    \item[\rm (iii)] $\operatorname{div}X=0$;
    \item[\rm (iv)]
    $
    \int_M \langle X,\nabla S\rangle\,dM\leq 0
    \quad\text{and}\quad
    \int_M \langle X,\nabla\lambda\rangle\,dM\geq 0.
    $
\end{enumerate}

Whenever these conditions hold,
$$
\nabla_XX=\frac{m(n-2)}{4}\nabla\lambda,
\qquad
\nabla S=\frac{n+2}{2}\nabla\lambda,
\qquad
\nabla|X|^2=-\frac{m(n-2)}{2}\nabla\lambda.
$$

Moreover, $S$ is constant if and only if $\lambda$ is constant.
Additionally, $|X|$ is constant when $n=2$, whereas, for $n\geq3$,
$|X|$ is constant if and only if $\lambda$ is constant.
\end{corollary}
\begin{proof}
We prove equivalence through the cyclic chain
$\mathrm{(i)}\Rightarrow\mathrm{(ii)}\Rightarrow
\mathrm{(iii)}\Rightarrow\mathrm{(iv)}\Rightarrow\mathrm{(i)}$.

The implication $\mathrm{(i)}\Rightarrow\mathrm{(ii)}$ is precisely
Corollary~\ref{cor:phi-killing}. If $X$ is Killing, then
$\operatorname{div}X=0$, which proves
$\mathrm{(ii)}\Rightarrow\mathrm{(iii)}$. Assume that $\mathrm{(iii)}$ holds. Since $M$ is closed, applying \eqref{divergence} with $f=S$ and $f=\lambda$, respectively, gives
\begin{align*}
\int_M\langle X,\nabla S\rangle dM
&=
-\int_M S\operatorname{div}X dM
=0,\
\int_M\langle X,\nabla\lambda\rangle dM
=
-\int_M \lambda\operatorname{div}X dM
=0.
\end{align*}

Thus $\mathrm{(iii)}\Rightarrow\mathrm{(iv)}$.

Conversely, if $\mathrm{(iv)}$ holds, then, recalling that
$\Phi=2S-(n+2)\lambda$, we obtain
\begin{align*}
\int_M\langle X,\nabla\Phi\rangle dM
&=
2\int_M\langle X,\nabla S\rangle dM
-(n+2)\int_M\langle X,\nabla\lambda\rangle dM
\leq 0.
\end{align*}

Hence $\mathrm{(iv)}\Rightarrow\mathrm{(i)}$, completing the cyclic
chain and proving the equivalence of the four conditions.

Under any of these equivalent conditions, $X$ is Killing. Therefore,
\cite[Lemma~7, equation~(11)]{de2026generalized} yields
$ 
\nabla_XX=\frac{m(n-2)}{4}\nabla\lambda.
$
Moreover, Remark~\ref{rem:killing-gradient-relations} gives
$
 \nabla S=\frac{n+2}{2}\nabla\lambda.
$
Since $X$ is Killing, we also have
$
 \nabla|X|^2=-2\nabla_XX
 =-\frac{m(n-2)}{2}\nabla\lambda.
$

Consequently, $S$ is constant precisely when $\lambda$ is constant. If $n=2$, the preceding identity immediately shows that $|X|$ is constant. If $n\geq3$, then $m(n-2)/2$ is nonzero, since $m\neq0$, and the same identity shows that $|X|$ is constant precisely when $\lambda$ is constant.

\end{proof}

The preceding results immediately yield the following for the non-generalized
setting.

\begin{corollary}\label{cor:classical-results-section2}
Let $(M^n,g,X,\lambda)$ be a closed $m$-quasi-Einstein manifold. Then the following conditions are
equivalent:
\begin{enumerate}
    \item[\rm (i)] $\int_M\langle X,\nabla S\rangle\,dM\leq0$;
    \item[\rm (ii)] $X$ is Killing;
    \item[\rm (iii)] $\operatorname{div}X=0.$ 
\end{enumerate}

\end{corollary}

We point out that the present corollary gathers several previous results. Bahuaud et al.~\cite{Bahuaud2024} established, for $m\neq-2$, the equivalence between the Killing and divergence-free conditions, namely ${\rm (ii)}$ and ${\rm (iii)}$. This restriction on $m$ was subsequently removed by Cochran~\cite[Corollary~1.4]{Cochran2025}, who proved the same equivalence for arbitrary admissible $m$. Cochran~\cite{Cochran2025} also showed, for $m\neq-2$, that constant scalar curvature is equivalent to the Killing condition, while Costa--Filho~\cite{CostaFilho2024} obtained the equivalence among constant scalar curvature, the Killing condition, and the divergence-free condition for arbitrary admissible $m$. The equivalence between constant scalar curvature and the Killing condition had already been established by Ghosh~\cite[Theorem~4.2 and equation~(5.18)]{Gho20}. More recently, Sharma~\cite[Theorem~1.1]{Rsharma2025} replaced the constant scalar curvature condition by the weaker integral condition in ${\rm (i)}$, proving that it implies ${\rm (ii)}$, and hence in view of the preceding characterizations, that the scalar curvature is constant.

\section{Differential identities and rigidity for negative  \texorpdfstring{$m$}{m}} \label{sec:Witten}
This section develops the identities and the general
mechanisms underlying our rigidity results for negative $m$. We first
recast the integrated Bochner formula as a Witten-type Hodge-energy
identity for generalized $m$-quasi-Einstein manifolds, which singles out
$m=-2$ through the cancellation of its quartic term, and we derive a
drift-Laplacian identity yielding triviality for $m\leq-n$; both results
allow nonconstant $\lambda$ under sign conditions on
$\langle X,\nabla\lambda\rangle$. We then turn to the constant-$\lambda$
case and establish differential identities for twisted exterior
derivatives and codifferentials, which lead to a mixed twisted operator.
The value $m=-4$ is the one at which this operator is self-adjoint, and
the corresponding rigidity theorem is proved here. The same identities,
combined with a principal-eigenvalue argument, yield
Theorem~\ref{thm:main} in Section~\ref{sec:proof-main}.

\subsection{A Witten-type Hodge-energy identity}
Let $d$ denote the exterior derivative and $\delta$ its formal adjoint. For two-forms, our norm
convention is
\begin{equation}\label{eq:two-form-norm}
|dX^\flat|^2
:=
\frac12 (dX^\flat)_{ij}(dX^\flat)^{ij}.
\end{equation}

With this convention and using the known orthogonality, a straightforward calculation yields
\begin{equation}\label{eq:nablaX-decomposition}
|\nabla X|^2
=
\frac14 |\mathcal L_X g|^2
+
\frac12 |dX^\flat|^2.
\end{equation} 

Since $\delta X^\flat=-\operatorname{div}X$, the integrated Hodge-Weitzenb\"ock formula for $X^\flat$ gives
\begin{equation}\label{eq:wh-hodge-weitzenbock}
\int_M\left(|dX^\flat|^2+(\operatorname{div}X)^2\right)\,dM
=
\int_M\left(|\nabla X|^2+\operatorname{Ric}(X,X)\right)\,dM.
\end{equation}

We point out that \eqref{eq:wh-hodge-weitzenbock} is nothing more than the Yano formula
\eqref{eq:bochner} combined with orthogonal decomposition \eqref{eq:nablaX-decomposition}.

We now combine the differential identity in Lemma~\ref{lem:laplacianX}
with the Hodge-Weitzenb\"ock formula for $X^\flat$, to show that

\begin{theorem}\label{thm:witten-hodge}
Let $(M^n,g,X,\lambda)$ be a closed generalized $m$-quasi-Einstein manifold. Then
\begin{equation}\label{eq:witten-hodge}
\begin{aligned}
\int_M\left[
|dX^\flat|^2+
\left(
\operatorname{div}X-\frac{m-2}{2m}|X|^2
\right)^2
\right]\,dM
&=
\frac{(m+2)^2}{4m^2}\int_M|X|^4\,dM \\
&\quad
+2\int_M\lambda|X|^2\,dM
+(n-2)\int_M\langle X,\nabla\lambda\rangle\,dM.
\end{aligned}
\end{equation}
\end{theorem}

\begin{proof}
By \eqref{eq:wh-hodge-weitzenbock},
\begin{equation}\label{eq:wh-split}
\begin{aligned}
\int_M\left(|dX^\flat|^2+(\operatorname{div}X)^2\right)\,dM
={}&
\int_M\left(|\nabla X|^2-\operatorname{Ric}(X,X)\right)\,dM\\
&+2\int_M\operatorname{Ric}(X,X)\,dM.
\end{aligned}
\end{equation}

The first term on the right-hand side is given by \eqref{eq:key}. On the one hand, contracting
\eqref{GQE} with $X$ gives
\begin{equation}\label{eq:Ric-on-X}
\operatorname{Ric}(X,X)
=\
\lambda|X|^2+\frac1m|X|^4-\frac12X(|X|^2).
\end{equation}

On the other hand, applying \eqref{divergence} with $f=|X|^2$, we have
$$
\int_M X(|X|^2)\,dM
=
-\int_M|X|^2\operatorname{div}X\,dM.
$$

Hence
$$
2\int_M\operatorname{Ric}(X,X)\,dM
=
2\int_M\lambda|X|^2\,dM
+\frac2m\int_M|X|^4\,dM
+\int_M|X|^2\operatorname{div}X\,dM.
$$

Using these in the second term of \eqref{eq:wh-split}, we obtain
$$
\begin{aligned}
\int_M\left[
|dX^\flat|^2+(\operatorname{div}X)^2
-\frac{m-2}{m}|X|^2\operatorname{div}X
\right]\,dM
={}&
\frac2m\int_M|X|^4\,dM\\
&+2\int_M\lambda|X|^2\,dM\\
&+(n-2)\int_M\langle X,\nabla\lambda\rangle\,dM.
\end{aligned}
$$

Completing the square, and since
$ 
\frac{(m-2)^2}{4m^2}+\frac2m
=
\frac{(m+2)^2}{4m^2}
$
yields \eqref{eq:witten-hodge}.
\end{proof}

Thus, we have the following triviality-type result.

\begin{corollary}\label{cor:m-minus-two}
Let $(M^n,g,X,\lambda)$ be a closed generalized $m$-quasi-Einstein
manifold, with $n\geq2$. If $m=-2,\,\, \lambda\leq0$ and $\int_M\langle X,\nabla\lambda\rangle\,dM\leq0,$
then $X\equiv0$.
\end{corollary}

\begin{proof}
Setting $m=-2$ in \eqref{eq:witten-hodge} gives
$$
\int_M\left[
|dX^\flat|^2+
\left(\operatorname{div}X-|X|^2\right)^2
\right]\,dM
=
2\int_M\lambda|X|^2\,dM
+(n-2)\int_M\langle X,\nabla\lambda\rangle\,dM.
$$

By assumption, the right-hand side is nonpositive, whereas the left-hand
side is nonnegative. Hence
$
\operatorname{div}X=|X|^2.
$
Integrating over the closed manifold $M$ yields
$$
0=\int_M\operatorname{div}X\,dM
=\int_M|X|^2\,dM,
$$
and therefore $X\equiv0$.
\end{proof}

\begin{remark}
Corollary~\ref{cor:m-minus-two} confirms the Colling and Dunajski conjecture (see \cite[p.197]{colling2026quasi}) in the case $m=-2$. Indeed, when $\lambda$ is constant and
$\lambda\leq0$, the condition
$ 
\int_M\langle X,\nabla\lambda\rangle\,dM\leq0
$
is automatic and therefore $X\equiv0$. For $n\leq4$, the value $m=-2$
lies in the range $m\leq2-n$ previously considered in \cite{colling2026quasi}, whereas for $n\geq5$ it lies outside that range.
Thus, the corollary proves the $m=-2$ case in every dimension and,
moreover, extends this phenomenon of rigidity to the generalized setting.
For constant $\lambda$, this is the endpoint case $m=-2$ of
Theorem~\ref{thm:main}, whose proof for $m<-2$ in
Section~\ref{sec:proof-main} relies on a different mechanism.
\end{remark}

For the case where the vector field $X$ admits an appropriate bound, we get the following 
consequence from Theorem \ref{thm:witten-hodge}.
\begin{corollary}
Let $(M^n,g,X,\lambda)$, $n\geq2$, be a closed generalized
$m$-quasi-Einstein manifold with $m<0$, $m\neq-2$, and $\lambda<0$.
If
$
\|X\|_{\infty}:=\max_M|X|,\,\,
\int_M\langle X,\nabla\lambda\rangle\,dM\leq0.
$
and
$
\|X\|_{\infty}^2
\leq
-\frac{8m^2}{(m+2)^2}\max_M\lambda,
$
then $X\equiv0$.
\end{corollary}

\begin{proof}
The assumptions make the right-hand side of \eqref{eq:witten-hodge}
nonpositive, whereas its left-hand side is nonnegative. Hence both
sides vanish, and therefore
$$
\operatorname{div}X=\frac{m-2}{2m}|X|^2.
$$

Integrating over $M$ yields $X\equiv0$.
\end{proof}

\subsection{Interpretation as a twisted Hodge energy}\label{sec:twisted-energy}
For $s\in\mathbb{R}$, define the differential operators $
 d_s\colon\Omega^k(M)\to\Omega^{k+1}(M)$ and  $\delta_s\colon\Omega^k(M)\to\Omega^{k-1}(M)
$ by

\begin{equation}\label{eq:wedge-twist}
d_s=d+sX^\flat\wedge,
\qquad
\delta_s=\delta+s\iota_X,
\end{equation}
where $X^\flat\wedge$ denotes exterior multiplication by $X^\flat$ and
$\iota_X$ denotes contraction with $X$. These operators have degrees
$+1$ and $-1$, respectively.

Since the formal adjoint of the exterior multiplication
by $X^\flat$ is contraction with $X$, the formal adjoint of $d_s$ is
 $
\delta_s=\delta+s\iota_X
 $, i.e, $\delta_s=d_s^*$.
Evaluating these operators at $X^\flat$, we obtain
 $
d_sX^\flat
=dX^\flat+sX^\flat\wedge X^\flat
=dX^\flat
 $
and
 $
\delta_sX^\flat
=\delta X^\flat+s\iota_XX^\flat
=-\operatorname{div}X+s|X|^2.
 $
Therefore,
 $$
\|d_sX^\flat\|_{L^2}^2+\|\delta_sX^\flat\|_{L^2}^2
=
\int_M\left[
|dX^\flat|^2+
\left(\operatorname{div}X-s|X|^2\right)^2
\right]\,dM.
 $$

We now choose
$ 
s=\frac{m-2}{2m}.
$
Then the left-hand side of~\eqref{eq:witten-hodge} is precisely
$
\|d_sX^\flat\|_{L^2}^2+\|\delta_sX^\flat\|_{L^2}^2.
 $
Consequently, identity~\eqref{eq:witten-hodge} can be written as
 $$
\begin{aligned}
\|d_sX^\flat\|_{L^2}^2+\|\delta_sX^\flat\|_{L^2}^2
&=
\frac{(m+2)^2}{4m^2}\int_M|X|^4\,dM
+2\int_M\lambda|X|^2\,dM \\
&\quad +(n-2)\int_M\langle X,\nabla\lambda\rangle\,dM.
\end{aligned}
 $$
 
This is the precise sense in which \eqref{eq:witten-hodge} is a twisted
Hodge-energy identity, in the spirit of Witten deformation~\cite[Section 3]{Witten}.

\subsection{A drift-Laplacian consequence}
We denote by $\Delta_X$ the drift Laplacian (also called diffusion operator or $X$–Laplacian in the literature) associated with $X$, defined on smooth functions $f\in C^\infty(M)$ by
$$
\Delta_X f=\Delta f-\langle X,\nabla f\rangle
=\Delta f-X(f).
$$

It is worth noting that in ~\cite[Section 3]{CatinoMastroliaMonticelliRigoli} Catino et al. consider this operator as the particular case $T=g$ of a more general second-order operator with drift. Combining \eqref{eq:Ric-on-X} with
Lemma~\ref{lem:laplacianX} yields
\begin{equation}\label{eq:drift-norm}
\begin{aligned}
\frac12\Delta_X|X|^2
={}&
|\nabla X|^2
-\lambda|X|^2
-\frac1m|X|^4
+\frac2m|X|^2\operatorname{div}X \\
&+(2-n)\langle X,\nabla\lambda\rangle.
\end{aligned}
\end{equation}

When $\lambda$ is constant, equation~\eqref{eq:drift-norm} reduces to equation~(2) of Lemma~2 in Barros and Ribeiro~\cite{BarrosRibeiro2012}. Hence, \eqref{eq:drift-norm} can be viewed as an extension of that identity to the generalized $m$-quasi-Einstein setting.

Next, using the decomposition (\ref{eq:nablaX-decomposition}) the equation \eqref{eq:drift-norm} becomes
$$
\begin{aligned}
\frac12\Delta_X|X|^2
={}&
\frac14|\mathcal L_Xg|^2
+\frac12|dX^\flat|^2
-\lambda|X|^2
-\frac1m|X|^4 \\
&+\frac2m|X|^2\operatorname{div}X
+(2-n)\langle X,\nabla\lambda\rangle.
\end{aligned}
$$ 

Taking the trace-free part of $\mathcal L_Xg$ and completing the
square involving $\operatorname{div}X$, we arrive at the following identity
\begin{equation}\label{eq:drift-square-completion}
\begin{aligned}
\frac12\Delta_X|X|^2
={}&
\frac14
\left|
\mathcal L_Xg-\frac2n(\operatorname{div}X)g
\right|^2
+\frac12|dX^\flat|^2 \\
&+
\frac1n
\left(
\operatorname{div}X+\frac{n}{m}|X|^2
\right)^2
-\lambda|X|^2
-\frac{m+n}{m^2}|X|^4 \\
&+(2-n)\langle X,\nabla\lambda\rangle.
\end{aligned}
\end{equation}

Thus, taking into account this expression, we derive  the following triviality result.

\begin{corollary}\label{cor:drift-triviality}
Let $(M^n,g,X,\lambda)$, $n\geq2$, be a closed generalized
$m$-quasi-Einstein manifold with $\lambda\leq0$ and $ m\leq-n\,$. When $n\geq3$, suppose that
$ \langle X,\nabla\lambda\rangle\leq0\,\,\text{on }M.$ Then $X\equiv0$.
\end{corollary}

\begin{proof}
Under the assumptions above, every term on the right-hand side of
\eqref{eq:drift-square-completion} is nonnegative. Hence
$
\Delta_X|X|^2\geq0.
$
Since $M$ is closed, the strong maximum principle applied to the drift operator $\Delta_X$ implies that $|X|^2$ is constant (see Pucci, Rigoli and Serrin~\cite[Theorem~5.6]{PucciRigoliSerrin} and also Catino et al.~\cite[Section~3]{CatinoMastroliaMonticelliRigoli} for related maximum principle results for drifted operators). It then follows from \eqref{eq:drift-square-completion} that each of the nonnegative terms on its right-hand side vanishes identically.

If $m<-n$, then
$
-\frac{m+n}{m^2}|X|^4=0,
$
and therefore $X\equiv0$. If $m=-n$, the vanishing of the divergence
square gives
$
\operatorname{div}X=|X|^2.
$
Integrating this identity over $M$, we obtain
$$
0=\int_M\operatorname{div}X\,dM
  =\int_M|X|^2\,dM,
$$
and hence $X\equiv0$.
\end{proof}

\begin{remark}
When $n=2$, the term involving $\nabla\lambda$ in
\eqref{eq:drift-square-completion} vanishes, so no additional
assumption on $\langle X,\nabla\lambda\rangle$ is required. For constant $\lambda$, the range $m\leq-n$ is contained in the
range $m\leq2-n$ covered by Colling and Dunajski \cite{colling2026quasi}. Thus,
Corollary~\ref{cor:drift-triviality} does not enlarge the known
parameter range in the classical $m$-quasi-Einstein setting.
However, it does extend the corresponding triviality result to
generalized $m$-quasi-Einstein manifolds under the pointwise sign
condition
$ 
\langle X,\nabla\lambda\rangle\leq0
$ on $M^n$.
\end{remark}

\subsection{Codifferential identities and twisted operators}\label{sec:wedge-twisted-factorization}

The twisted Hodge identity obtained in the preceding section singles out
the value $m=-2$ through the cancellation of its quartic term. We now
develop a different approach, based on the twisted operators
\eqref{eq:wedge-twist}, that identifies $m=-4$ as a
distinguished self-adjoint case and, more generally, provides the
differential identities on which the proof of Theorem~\ref{thm:main} in
Section~\ref{sec:proof-main} rests. For clarity, we
assume throughout the remainder of this section and in Section~\ref{sec:proof-main} that $\lambda$ is constant. 
From \eqref{eq:wedge-twist}, one immediately obtains
\begin{equation}\label{eq:wedge-twist-square}
 d_s^2=s\,dX^\flat\wedge,
 \qquad
 \delta_s^2=s\,(dX^\flat\wedge)^*.
\end{equation}

We first derive an identity for $\delta dX^\flat$ by combining
the trace and divergence of the $m$-quasi-Einstein equation
with the Hodge--Weitzenb\"ock formula. We then express this
identity in terms of the twisted operators.

\begin{lemma}\label{lem:intrinsic-codifferential}
Let $(M^n,g,X,\lambda)$ be a $m$-quasi-Einstein
manifold. Then
\begin{equation}\label{eq:intrinsic-codifferential}
\begin{aligned}
  \delta dX^\flat
  ={}&-\frac{m+2}{m}\,\iota_XdX^\flat
     +d\bigl(\operatorname{div}X-|X|^2\bigr) \\
    &-\frac2m\bigl(\operatorname{div}X-|X|^2\bigr)X^\flat
     +2\lambda X^\flat.
\end{aligned}
\end{equation}
\end{lemma}

\begin{proof}
In addition to the previously established identity \eqref{div XX}, we shall use

$$
  \operatorname{div}(\mathcal L_Xg)
  =\operatorname{tr}_g\nabla^2X^\flat
   +d(\operatorname{div}X)+\operatorname{Ric}(X,\cdot).
$$

Since $\lambda$ is constant throughout this section, combining this formula with
\eqref{eq:divergence-section2} and the differential of the trace identity
\eqref{eq2} yields
\begin{equation}\label{eq:divergence-section-six}
\begin{aligned}
0={}&\operatorname{tr}_g\nabla^2X^\flat
+\operatorname{Ric}(X,\cdot)
+\frac1m d|X|^2 -\frac2m(\operatorname{div}X)X^\flat
-\frac2m(\nabla_XX)^\flat.
\end{aligned}
\end{equation}

On the other hand, the Hodge--Weitzenb\"ock formula and
$\delta X^\flat=-\operatorname{div}X$ yield
\begin{equation}\label{eq:hodge-section-six}
  \delta dX^\flat
  =-\operatorname{tr}_g\nabla^2X^\flat
   +d(\operatorname{div}X)+\operatorname{Ric}(X,\cdot).
\end{equation}

Moreover, evaluating \eqref{GQE} on $X$ gives
\begin{equation}\label{eq:ricci-on-x-section-six}
  2\operatorname{Ric}(X,\cdot)
  =2\lambda X^\flat-(\nabla_XX)^\flat
   -\frac12d|X|^2+\frac2m|X|^2X^\flat.
\end{equation}

Finally, the definition of the exterior derivative gives
\begin{equation}\label{eq:intrinsic-nabla-x}
  (\nabla_XX)^\flat
  =\iota_XdX^\flat+\frac12d|X|^2.
\end{equation}

Substituting \eqref{eq:hodge-section-six} and
\eqref{eq:ricci-on-x-section-six} into
\eqref{eq:divergence-section-six}, and then using
\eqref{eq:intrinsic-nabla-x} to eliminate $(\nabla_XX)^\flat$, gives
\eqref{eq:intrinsic-codifferential}.
\end{proof}


The preceding codifferential identity can now be reorganized as a scalar
equation involving two different twisting parameters. This is the key
factorization from which the exceptional value of $m$ will emerge.

\begin{proposition}\label{prop:mixed-twisted-factorization}
Let $(M^n,g,X,\lambda)$ be a $m$-quasi-Einstein
manifold. Then,
\begin{equation}\label{eq:mixed-twisted-factorization}
  \delta_{\frac{m+2}{m}}d_{-\frac2m}
  \left(
    \lambda+\frac{|X|^2-\operatorname{div}X}{m}
  \right)
  =-\frac{m+2}{m^2}|dX^\flat|^2.
\end{equation}
\end{proposition}

\begin{proof}
Adding $\frac{m+2}{m}\iota_XdX^\flat$ to both sides of
\eqref{eq:intrinsic-codifferential}, we find
\begin{equation}\label{eq:first-twisted-step}
\begin{aligned}
  \delta_{\frac{m+2}{m}}dX^\flat
  ={}&d\bigl(\operatorname{div}X-|X|^2\bigr) \\
     &-\frac2m\bigl(\operatorname{div}X-|X|^2\bigr)X^\flat
      +2\lambda X^\flat.
\end{aligned}
\end{equation}

Since $\lambda$ is constant, the definition of $d_{-2/m}$ shows that
\eqref{eq:first-twisted-step} is equivalent to
\begin{equation}\label{eq:second-twisted-step}
  d_{-\frac2m}
  \left(
    \lambda+\frac{|X|^2-\operatorname{div}X}{m}
  \right)
  =-\frac1m\delta_{\frac{m+2}{m}}dX^\flat.
\end{equation}

Using first \eqref{eq:wedge-twist-square} and the norm convention
\eqref{eq:two-form-norm}, we have
$$
\delta_s^2dX^\flat
=
s(dX^\flat\wedge)^*(dX^\flat)
=
s|dX^\flat|^2.
$$

Applying now $\delta_{(m+2)/m}$ to \eqref{eq:second-twisted-step} and taking
$s=(m+2)/m$, we obtain \eqref{eq:mixed-twisted-factorization}.
\end{proof}

The significance of the two twisting parameters becomes transparent
through the quadratic form associated with the mixed operator. Define
$$
\mathcal Q_m(f)
:=
\left\langle
f,\delta_{\frac{m+2}{m}}d_{-\frac2m}f
\right\rangle_{L^2(M)}
=
\int_M
f\,\delta_{\frac{m+2}{m}}d_{-\frac2m}f\,dM,
\qquad f\in C^\infty(M).
$$

Since
$\delta_{\frac{m+2}{m}}
=d_{\frac{m+2}{m}}^*$, adjointness and \eqref{eq:wedge-twist} give
$$ 
\begin{aligned}
\mathcal Q_m(f)
&=
\int_M
\left\langle
d_{\frac{m+2}{m}}f,
d_{-\frac2m}f
\right\rangle dM\\
&=
\int_M
\left\langle
df+\frac{m+2}{m}fX^\flat,
df-\frac2m fX^\flat
\right\rangle dM\\
&=
\int_M
\left[
|df|^2+f\langle X,\nabla f\rangle
-\frac{2(m+2)}{m^2}|X|^2f^2
\right]dM.
\end{aligned}
$$

Completing the square, we obtain
\begin{equation}\label{eq:completed-mixed-quadratic-form}
\mathcal Q_{m}(f)
  =\|d_{1/2}f\|_{L^2}^2
   -\frac{(m+4)^2}{4m^2}\int_M|X|^2f^2\,dM.
\end{equation}

The correction term in \eqref{eq:completed-mixed-quadratic-form}
vanishes precisely at $m=-4$, the unique value for which the two
twisting parameters coincide. Then both are equal to $1/2$, and
$
\delta_{1/2}d_{1/2}=d_{1/2}^*d_{1/2}\geq0.
$
For $m\neq-4$ the two twisting parameters differ, the mixed
operator $\delta_{\frac{m+2}{m}}d_{-\frac2m}$ is not self-adjoint, and
the quadratic form \eqref{eq:completed-mixed-quadratic-form} carries a
negative correction term. The sign information needed for rigidity must
then be extracted differently; this is done in
Section~\ref{sec:proof-main}, for every $m<-2$, by means of a weighted
energy identity for a positive principal eigenfunction of the mixed
operator. We first treat the self-adjoint case.
\subsection{Rigidity at \texorpdfstring{$m=-4$}{m=-4}: the self-adjoint case}

Setting $m=-4$ in Proposition~\ref{prop:mixed-twisted-factorization}
and \eqref{eq:completed-mixed-quadratic-form} yields
\begin{equation}\label{eq:m-minus-four-scalar-identity}
  \delta_{1/2}d_{1/2}
  \left[
    \lambda+\frac14\bigl(\operatorname{div}X-|X|^2\bigr)
  \right]
  =\frac18|dX^\flat|^2,
\end{equation}
and, for every $f\in C^\infty(M)$,
$$ 
  \mathcal Q_{-4}(f)
  =\|d_{1/2}f\|_{L^2}^2
  =\int_M\left|df+\frac12fX^\flat\right|^2\,dM.
$$

On functions, \eqref{eq:wedge-twist} gives
$$
  d_{1/2}f=df+\frac12fX^\flat,
  \qquad
  \delta_{1/2}=\delta+\frac12\iota_X.
$$

Using $\delta df=-\Delta f$ and
$
  \delta(fX^\flat)
  =f\,\delta X^\flat-\iota_{\nabla f}X^\flat
  =-f\operatorname{div}X-X(f),
 $
we obtain
\begin{align}
  \delta_{1/2}d_{1/2}f
  &=
  \left(\delta+\frac12\iota_X\right)
  \left(df+\frac12fX^\flat\right) \notag\\
  &=
  -\Delta f-\frac12X(f)
  -\frac12(\operatorname{div}X)f
  +\frac12X(f)+\frac14|X|^2f \notag\\
  &=
  -\Delta f-\frac12(\operatorname{div}X)f
  +\frac14|X|^2f.
  \label{eq:m-minus-four-operator}
\end{align}

This is the case $a=b=\tfrac12$ of the general formula
\eqref{eq:operator-general} for the operators $P_{a,b}=\delta_ad_b$
given in Section~\ref{sec:proof-main}.
We are now ready to establish the following result.

\begin{theorem}\label{thm:rigidity-m-minus-four}
Let $(M^n,g,X,\lambda)$ be a closed $m$-quasi-Einstein manifold with
constant $\lambda\leq0$. If $m=-4$, then $X\equiv0$.
\end{theorem}

\begin{proof}
If $X^\flat$ is exact, then $X$ is a gradient vector field, and as observed by Colling and Dunajski immediately before
\cite[Proposition~2.4]{colling2026quasi}, in the gradient case their
formula~(2.8) reduces to the warped-product identity of Kim and Kim
\cite[equation~(1.3)]{KimKim2003}. For $\lambda\leq0$, the
corresponding maximum-principle argument yields $X\equiv0$ (see also Lemma~\ref{lem:gradient} below, where a self-contained proof based on \eqref{eq:second-twisted-step} is given for every $m<0$). We may therefore assume that
$X^\flat$ is not exact.

We first claim that
\begin{equation}\label{eq:trivial-kernel}
    \ker(\delta_{1/2}d_{1/2})=\{0\}.
\end{equation}

Indeed, if $\delta_{1/2}d_{1/2}h=0$, then
$$
0
=
\left\langle h,\delta_{1/2}d_{1/2}h\right\rangle_{L^2(M)}
=
\mathcal Q_{-4}(h)
=
\|d_{1/2}h\|_{L^2(M)}^2.
$$

Hence $d_{1/2}h=0$, that is,
$ 
dh+\frac12hX^\flat=0.
$
If $h$ vanishes at one point, the preceding first-order equation
along smooth curves implies that $h\equiv0$. Otherwise $h$ is
nowhere zero and
$ 
X^\flat=-2d\log|h|,
$
contradicting the assumption that $X^\flat$ is not exact. 
Since $\delta_{1/2}d_{1/2}$ is elliptic, self-adjoint, and
nonnegative, \eqref{eq:trivial-kernel} implies that its lowest
eigenvalue $\mu_1$ is strictly positive. Let $\varphi>0$ be a
corresponding eigenfunction:
$$
\delta_{1/2}d_{1/2}\varphi=\mu_1\varphi.
$$

For every $f\in C^\infty(M)$, applying
~\eqref{eq:m-minus-four-operator} to $\varphi f$ and using the product rule
for the Laplacian, we obtain
\begin{align*}
\delta_{1/2}d_{1/2}(\varphi f)
&=
-\Delta(\varphi f)
-\frac12(\operatorname{div}X)\varphi f
+\frac14|X|^2\varphi f\\
&=
-\varphi\Delta f
-2\langle\nabla\varphi,\nabla f\rangle
+f\left(
-\Delta\varphi
-\frac12(\operatorname{div}X)\varphi
+\frac14|X|^2\varphi
\right)\\
&=
-\varphi\Delta f
-2\langle\nabla\varphi,\nabla f\rangle
+\mu_1\varphi f.
\end{align*}

Since $\nabla\varphi=\varphi\nabla\log\varphi$, this becomes
\begin{equation}\label{eq:ground-state-transform}
\delta_{1/2}d_{1/2}(\varphi f)
=
\varphi\left(
-\Delta f
-2\langle\nabla\log\varphi,\nabla f\rangle
+\mu_1f
\right).
\end{equation}

Set
$$
f=
\frac{\lambda+\frac14(\operatorname{div}X-|X|^2)}{\varphi}.
$$

Suppose that $f$ attained a negative minimum, then at that point
$\nabla f=0$ and $\Delta f\geq0$.

Hence
\eqref{eq:ground-state-transform} would give
$$
\delta_{1/2}d_{1/2}(\varphi f)<0,
$$
because $\mu_1f<0$. This contradicts \eqref{eq:m-minus-four-scalar-identity}, whose
right-hand side is nonnegative. Therefore $f\geq0$.

Since
$\varphi>0$, it follows that
$$
\lambda+\frac14(\operatorname{div}X-|X|^2)\geq0
\qquad\text{on }M.
$$

Integrating over $M$, we obtain
$$
0\leq
\lambda\operatorname{Vol}(M)
-\frac14\int_M|X|^2\,dM
\leq0.
$$

Hence $\lambda=0$ and $X\equiv0$, contradicting the assumption that
$X^\flat$ is not exact. This completes the proof.
\end{proof}

\section{\texorpdfstring{Proof of Theorem~\ref{thm:main}}{Proof of the main theorem}}\label{sec:proof-main}

Throughout this section, we keep the twisted
operators $d_s$ and $\delta_s$ of \eqref{eq:wedge-twist} and the norm
convention \eqref{eq:two-form-norm}, and we write
$$
\langle\alpha,\beta\rangle_{L^2}:=\int_M\langle\alpha,\beta\rangle\,dM,
\qquad
\|\alpha\|_{L^2}^2:=\langle\alpha,\alpha\rangle_{L^2},
$$
for real differential forms $\alpha,\beta$ of the same degree, functions
included. As before, $\Delta=\operatorname{div}\nabla$, so that
$\delta df=-\Delta f$.

\subsection{The operators \texorpdfstring{$P_{a,b}$}{P(a,b)}}
For $a,b\in\mathbb R$, consider the second-order operator
$P_{a,b}:=\delta_ad_b$ acting on functions. Using $\delta df=-\Delta f$
and
$\delta(fX^\flat)=f\,\delta X^\flat-\iota_{\nabla f}X^\flat
=-f\operatorname{div}X-X(f)$, exactly as in the derivation of
\eqref{eq:m-minus-four-operator}, we obtain
\begin{align}
P_{a,b}f
&=(\delta+a\iota_X)(df+bfX^\flat)\notag\\
&=-\Delta f-bX(f)-b(\operatorname{div}X)f+aX(f)+ab|X|^2f\notag\\
&=-\Delta f+(a-b)X(f)-b(\operatorname{div}X)f+ab|X|^2f.
\label{eq:operator-general}
\end{align}
The family $P_{a,b}$ naturally contains the elliptic operator underlying the Colling--Dunajski construction as a particular case. Indeed, for $a=0$ and $b=1$,
$$
P_{0,1}f
=-\Delta f-X(f)-(\operatorname{div}X)f
=-\bigl(\Delta f+\operatorname{div}(fX)\bigr).
$$
Thus the positive function $\Gamma$ used in the Colling--Dunajski argument, see equation (2.6) in \cite{colling2026quasi}, characterized by
$$
\Delta\Gamma+\operatorname{div}(\Gamma X)=0,
$$
is precisely a positive element of the kernel of $P_{0,1}$.

Since $\delta_s=d_s^*$, the formal $L^2$-adjoint of $P_{a,b}$ is
\begin{equation}\label{eq:adjoint-general}
P_{a,b}^*=(\delta_ad_b)^*=d_b^*\delta_a^*=\delta_bd_a=P_{b,a}.
\end{equation}

In particular, $P_{a,b}$ is self-adjoint precisely when $a=b$ or $X$ is null.

The operator $P_{a,b}$ has real smooth coefficients and is uniformly
elliptic, since its second-order part is $-\Delta$. Thus, it has a real
principal eigenvalue $\mu_1$ and a smooth positive eigenfunction $\phi$,
unique up to multiplication by a positive constant; moreover, an
eigenvalue admitting a positive eigenfunction is necessarily the
principal one. Its formal adjoint $P_{b,a}$ has the same principal
eigenvalue and also admits a positive eigenfunction. These facts follow
from \cite[Lemma~4.1 and Appendix~B]{andersson2008stability}; see also
\cite{galloway2018rigidity} for a concise exposition.

The mixed operator of Proposition~\ref{prop:mixed-twisted-factorization}
is the case
\begin{equation}\label{eq:Pm-def}
\mathcal P_m:=P_{\frac{m+2}{m},-\frac2m}
=\delta_{\frac{m+2}{m}}d_{-\frac2m},
\end{equation}
so that the quadratic form of
Section~\ref{sec:wedge-twisted-factorization} is
$\mathcal Q_m(f)=\langle f,\mathcal P_mf\rangle_{L^2}$. By
\eqref{eq:operator-general} with $a=(m+2)/m$ and $b=-2/m$,
\begin{equation}\label{eq:Pm}
\mathcal P_mf
=-\Delta f+\frac{m+4}{m}\langle X,\nabla f\rangle
+\left(\frac2m\operatorname{div}X-\frac{2(m+2)}{m^2}|X|^2\right)f,
\end{equation}
which reduces to \eqref{eq:m-minus-four-operator} at $m=-4$. Setting
\begin{equation}\label{eq:F-def}
F:=\lambda+\frac{|X|^2-\operatorname{div}X}{m},
\end{equation}
identity \eqref{eq:mixed-twisted-factorization} reads
\begin{equation}\label{eq:mixed}
\mathcal P_mF=-\frac{m+2}{m^2}|dX^\flat|^2 .
\end{equation}

In particular, $\mathcal P_mF\geq0$ on $M$ whenever $m<-2$.

\subsection{A principal-eigenvalue identity}
The following lemma is the key new ingredient. It shows that, although
$P_{a,b}$ is not self-adjoint when $a\neq b$, pairing the eigenvalue
equation with a suitable power of the positive eigenfunction still
produces a weighted energy identity with a definite sign.

\begin{lemma}\label{lem:power}
Let $a,b\in\mathbb R\setminus\{0\}$, set $q:=a/b$, and let
$P_{a,b}=\delta_ad_b$. If $P_{a,b}\phi=\mu_1\phi$ with $\phi>0$, then
\begin{equation}\label{eq:power}
\langle P_{a,b}\phi,\phi^q\rangle_{L^2}
=q\left\|\phi^{(q-1)/2}d_b\phi\right\|_{L^2}^2 .
\end{equation}
Consequently, if $X^\flat$ is not exact, then $\mu_1>0$ for $q>0$ and
$\mu_1<0$ for $q<0$. If $X^\flat$ is exact, then $\mu_1=0$.
\end{lemma}

\begin{proof}
Since $qb=a$, we have
$$
d_a(\phi^q)=q\phi^{q-1}d\phi+a\phi^qX^\flat
=q\phi^{q-1}\bigl(d\phi+b\phi X^\flat\bigr)
=q\phi^{q-1}d_b\phi .
$$

Pairing the eigenvalue equation with $\phi^q$ and using
\eqref{eq:adjoint-general} together with $\delta_b=d_b^*$, we obtain
\begin{align*}
\mu_1\langle\phi,\phi^q\rangle_{L^2}
&=\langle P_{a,b}\phi,\phi^q\rangle_{L^2}
=\langle\phi,P_{b,a}\phi^q\rangle_{L^2}
=\langle\phi,\delta_bd_a\phi^q\rangle_{L^2}\\
&=\langle d_b\phi,d_a\phi^q\rangle_{L^2}
=q\langle d_b\phi,\phi^{q-1}d_b\phi\rangle_{L^2}
=q\left\|\phi^{(q-1)/2}d_b\phi\right\|_{L^2}^2,
\end{align*}
which proves \eqref{eq:power}.

Since $\phi>0$, we have
$\langle\phi,\phi^q\rangle_{L^2}=\int_M\phi^{q+1}\,dM>0$. Moreover, the
squared norm in \eqref{eq:power} vanishes if and only if
$d_b\phi\equiv0$, that is, if and only if
$$
X^\flat=-\frac1b\,d(\log\phi).
$$

Thus, if $X^\flat$ is not exact, the right-hand side of \eqref{eq:power}
is strictly positive when $q>0$ and strictly negative when $q<0$, so
$\mu_1$ has the sign of $q$.

Finally, if $X^\flat=df$ for some $f\in C^\infty(M)$, the positive
function $\phi_0:=e^{-bf}$ satisfies
$d_b\phi_0=-be^{-bf}\,df+be^{-bf}\,df=0$, hence $P_{a,b}\phi_0=0$. Since
an eigenvalue admitting a positive eigenfunction is the principal
eigenvalue, we conclude that $\mu_1=0$.
\end{proof}

\begin{remark}\label{rem:power-m-minus-four}
At $m=-4$, the parameters of $\mathcal P_m$ satisfy $a=b=\tfrac12$ and
$q=1$, and \eqref{eq:power} becomes the quadratic energy identity
$\langle\mathcal P_{-4}\phi,\phi\rangle_{L^2}=\|d_{1/2}\phi\|_{L^2}^2$,
which is the mechanism behind \eqref{eq:trivial-kernel} in the proof of
Theorem~\ref{thm:rigidity-m-minus-four}. For every $m<-2$, the positive
power $q=-(m+2)/2$ provides the required sign, even though $\mathcal P_m$
is not self-adjoint.
\end{remark}

\subsection{The gradient case}
The gradient case is known (see \cite[Proposition~2.4]{colling2026quasi}),
where the result is attributed to Kim and Kim~\cite{KimKim2003}. We include
a short proof based on the first-order identity
\eqref{eq:second-twisted-step}, valid for every $m<0$.

\begin{lemma}\label{lem:gradient}
Let $(M^n,g,X,\lambda)$ be a closed $m$-quasi-Einstein
manifold with $m<0$ and constant $\lambda\leq0$. If $X^\flat$ is exact,
then $X\equiv0$.
\end{lemma}

\begin{proof}
Write $X^\flat=df$ and let $F$ be as in \eqref{eq:F-def}. Since
$dX^\flat=0$, identity \eqref{eq:second-twisted-step} gives
$$
0=d_{-\frac2m}F=dF-\frac2mF\,df=e^{2f/m}\,d\bigl(e^{-2f/m}F\bigr),
$$
so $F=ce^{2f/m}$ for some constant $c\in\mathbb R$. Setting
$h:=e^{-f/m}>0$, we have $F=ch^{-2}$, $\nabla f=-m\,\nabla h/h$, and
therefore
$$
|X|^2=\frac{m^2}{h^2}|\nabla h|^2,
\qquad
\operatorname{div}X=\Delta f=-\frac mh\Delta h+\frac{m}{h^2}|\nabla h|^2 .
$$

Substituting these expressions into \eqref{eq:F-def} and multiplying by
$h^2$, we obtain
\begin{equation}\label{eq:gradient}
c=h\Delta h+(m-1)|\nabla h|^2+\lambda h^2 .
\end{equation}

At a minimum point and at a maximum point of $h$, the gradient term
vanishes, while $h\Delta h$ is nonnegative and nonpositive,
respectively, because $h>0$. Hence \eqref{eq:gradient} yields
$$
\lambda h_{\min}^2\leq c\leq\lambda h_{\max}^2 .
$$

If $\lambda<0$, then $\lambda h_{\max}^2\leq\lambda h_{\min}^2$, and
these inequalities force $h_{\min}=h_{\max}$. If $\lambda=0$, they give
$c=0$, and \eqref{eq:gradient} becomes
$$
\Delta(h^m)=mh^{m-2}\bigl(h\Delta h+(m-1)|\nabla h|^2\bigr)=0,
$$
so $h^m$ is constant because $M$ is closed. In either case $h$, and
hence $f$, is constant and $X\equiv0$ as desired.
\end{proof}

\subsection{The non-gradient case}

We are now in a position to present the proof of the main Theorem.

\begin{proof}[Proof of Theorem~\ref{thm:main}]
Assume first that $m<-2$. By Lemma~\ref{lem:gradient}, it suffices to
exclude the case in which $X^\flat$ is not exact. The parameters of $\mathcal P_m$
in \eqref{eq:Pm-def}, namely $a=\frac{m+2}{m}$ and $b=-\frac2m$, are both
nonzero, and
$ 
q=\frac ab=-\frac{m+2}{2}>0 .
$
Hence Lemma~\ref{lem:power} shows that the principal eigenvalue $\mu_1$
of $\mathcal P_m$ is strictly positive. Let $\phi>0$ be a principal
eigenfunction, $\mathcal P_m\phi=\mu_1\phi$, and write $F=\phi f$ with
$f:=F/\phi\in C^\infty(M)$, where $F$ is given by \eqref{eq:F-def}.

Applying \eqref{eq:Pm} to $\phi f$ and using the product rule, we obtain
\begin{align*}
\mathcal P_m(\phi f)
&=-\Delta(\phi f)+\frac{m+4}{m}\langle X,\nabla(\phi f)\rangle
+\left(\frac2m\operatorname{div}X-\frac{2(m+2)}{m^2}|X|^2\right)\phi f\\
&=-\phi\Delta f-2\langle\nabla\phi,\nabla f\rangle
+\frac{m+4}{m}\phi\langle X,\nabla f\rangle\\
&\quad+f\left[-\Delta\phi+\frac{m+4}{m}\langle X,\nabla\phi\rangle
+\left(\frac2m\operatorname{div}X-\frac{2(m+2)}{m^2}|X|^2\right)\phi\right]\\
&=-\phi\Delta f-2\langle\nabla\phi,\nabla f\rangle
+\frac{m+4}{m}\phi\langle X,\nabla f\rangle+\mu_1\phi f,
\end{align*}
where we used $\mathcal P_m\phi=\mu_1\phi$ in the last step. Factoring
out $\phi>0$ yields the conjugation identity
\begin{equation}\label{eq:conjugation}
\mathcal P_m(\phi f)
=\phi\left[-\Delta f
+\left\langle\frac{m+4}{m}X-2\nabla\log\phi,\nabla f\right\rangle
+\mu_1f\right],
\end{equation}
which generalizes \eqref{eq:ground-state-transform}.

On the other hand, \eqref{eq:mixed} gives $\mathcal P_mF\geq0$ on $M$,
because $m<-2$. Since $M$ is closed, $f$ attains its minimum. If this
minimum were negative, then at a minimum point we would have
$\nabla f=0$, $\Delta f\geq0$, and $\mu_1f<0$, so that
\eqref{eq:conjugation} would give $\mathcal P_mF<0$ there, a
contradiction. Therefore $f\geq0$, and hence $F\geq0$ on $M$. Integrating
\eqref{eq:F-def}, we get
$$
0\leq\int_MF\,dM
=\lambda\operatorname{Vol}(M)+\frac1m\int_M|X|^2\,dM\leq0,
$$
since $\lambda\leq0$ and $m<0$. Both terms on the right-hand side must
therefore vanish, so $X\equiv0$. But then $X^\flat=0$ is exact,
contradicting our assumption. This proves the theorem for $m<-2$.

It remains to consider $m=-2$, for which $a=0$ and
Lemma~\ref{lem:power} does not apply. Since $\lambda$ is constant, the
condition $\int_M\langle X,\nabla\lambda\rangle\,dM\leq0$ holds
trivially, and Corollary~\ref{cor:m-minus-two} yields $X\equiv0$. We
note that the integral identity used there is recovered by taking the
$L^2$-pairing of \eqref{eq:intrinsic-codifferential} at $m=-2$ with
$X^\flat$, which gives \eqref{eq:witten-hodge} at $m=-2$ with constant
$\lambda$.
\end{proof}

\section{Appendix}

In this appendix, we present a closed
generalized $m$-quasi-Einstein manifold carrying a conformal
non-Killing potential vector field and satisfying
$ 
\int_M \langle X,\nabla\lambda\rangle\,dM<0.
$ 

\begin{theo}\label{thm:warped}
Let $q=n-1\geq2$ and $(F^q,g)$ be a closed Einstein manifold satisfying
$ 
\operatorname{Ric}=-(q-1)g,
$
and let $m>0$. For every $E>\min V$, where $V$ is the potential in
\eqref{eq:C6}, there exist $L>0$ and a smooth, positive, nonconstant
$L$-periodic function $\phi$ such that, on $M=\mathbb{S}^1_L\times F$,
$$
\bar g=dt^2+\phi(t)^2g,
\qquad
X=\phi(t)\partial_t,
$$
satisfy \eqref{GQE} for a smooth nonconstant function $\lambda$.
Moreover, $X$ is conformal and non-Killing, and
\begin{equation}\label{eq:C7}
\int_M\langle X,\nabla\lambda\rangle\,dM_{\bar g}
=-n\operatorname{Vol}(F,g)\int_0^L\phi^q(\phi')^2\,dt<0.
\end{equation}
\end{theo}

\begin{proof}
For vector fields $U,V$ tangent to $F$, the warped-product Ricci
formulas are (cf. \cite{o1983semi})
$$
\operatorname{Ric}_{\bar g}(\partial_t,\partial_t)
=-q\frac{\phi''}{\phi},
\qquad
\operatorname{Ric}_{\bar g}(\partial_t,U)=0,
$$
$$
\operatorname{Ric}_{\bar g}(U,V)
=
\left[-(q-1)-\phi\phi''-(q-1)(\phi')^2\right]g(U,V).
$$

Furthermore,
$\nabla_{\partial_t}X=\phi'\partial_t$ and $\nabla_UX=\phi'U$, so
$\frac12\mathcal L_X\bar g=\phi'\bar g,$ while
$X^\flat\otimes X^\flat=\phi^2dt^2$.  The radial and fiber components
of \eqref{GQE} are therefore
$$
\lambda=-q\frac{\phi''}{\phi}+\phi'-\frac{\phi^2}{m},
\qquad
\lambda=\frac{-(q-1)-\phi\phi''-(q-1)(\phi')^2}{\phi^2}+\phi'.
$$

They agree exactly when
\begin{equation}\label{eq:C5}
\phi\phi''=(\phi')^2+1-\frac{\phi^4}{m(q-1)}.
\end{equation}

The mixed components vanish identically.
Put $\psi=\log\phi$.  Equation \eqref{eq:C5} is
\begin{equation}\label{eq:C6}
\psi''=e^{-2\psi}-\frac{e^{2\psi}}{m(q-1)}=-V'(\psi),
\qquad
V(\psi)=\frac12e^{-2\psi}+
\frac{e^{2\psi}}{2m(q-1)}.
\end{equation}

Here $V''=2e^{-2\psi}+2e^{2\psi}/(m(q-1))>0$ and
$V(\psi)\to+\infty$ at both ends. It therefore has a unique critical point,
$
\psi_0=\frac14\log m(q-1),
$
which is its global minimum, with 
$V(\psi_0)=\frac1{\sqrt {m(q-1)}}.$ 

The total energy of the system is defined by
\begin{equation}\label{Energy}
\mathcal{E}(\psi,\psi')
=
\frac{1}{2}(\psi')^2+V(\psi),
\end{equation}
where the two terms on the right-hand side represent the kinetic and
potential energies, respectively. By the law of conservation of energy
\cite[p.~16]{Arnold1989}, $\mathcal{E}$ remains constant along every
solution. Indeed,
$$
\frac{d}{dt}\mathcal{E}(\psi,\psi')
=
\psi'\bigl(\psi''+V'(\psi)\bigr)
=
0.
$$ 

Moreover, for each energy value satisfying
$$
E>\min V=\frac{1}{\sqrt{m(q-1)}},
$$
the equation $V(\psi)=E$ has exactly two solutions
$ 
\psi_-<\psi_0<\psi_+.
$ 
These are the two simple turning points of the corresponding energy
level, since
$
V'(\psi_\pm)\neq0.
$   Consequently, the solution with energy $E$ oscillates between $\psi_-$
  and $\psi_+$. Solving \eqref{Energy} for $\psi'$ and separating
  variables, we obtain its period:
  $$
  L(E)
  =
  2\int_{\psi_-}^{\psi_+}
  \frac{ds}{\sqrt{2\bigl(E-V(s)\bigr)}}.
  $$
  
This integral is finite because the turning points are simple.
The local existence and uniqueness theorem for the smooth
autonomous system therefore ensures that the trajectory passes smoothly
from one monotone arc of $\Gamma_E = \mathcal{E}^{-1}(E)$ to the other at each turning point and returns to
its initial phase point after the finite time $L(E)$. Consequently,
the corresponding solution is nonconstant and $L(E)$-periodic. 
See, for instance, \cite[Section~6.7, pp.~132--135]{Teschl}
for the phase-plane and energy-level description of one-dimensional
Newton equations.

Consequently $\psi$, and hence $\phi=e^\psi$, is smooth, positive,
nonconstant, and periodic.  The metric, $X$, and the displayed
$\lambda$ therefore descend to $\mathbb{S}^1_L\times F$.

Finally, Theorem~\ref{1} applies because $X$ is conformal. Since
$\operatorname{div}X=n\phi'$ and $dM=\phi^qdt\,dV$, identity
\eqref{eq:C1} gives
$$
\int_M\langle X,\nabla\lambda\rangle\,dM
=-\frac1n\int_M(\operatorname{div}X)^2\,dM
=-n\operatorname{Vol}(F,g)\int_0^L\phi^q(\phi')^2\,dt.
$$

The last integral is positive for a nonconstant solution, proving
\eqref{eq:C7}.
\end{proof}

\textbf{Declarations}\\

\textbf{AI Disclosure.}
The proof of Theorem~\ref{thm:main} was suggested by
Astra~6.0, an OpenAI model accessed through ChatGPT.
In particular, the suggestion consisted of testing the principal
eigenfunction of $\delta_ad_b$ against the power $\phi^{a/b}$
(Lemma~\ref{lem:power}) and combining the resulting identity with the
differential identities of Section~\ref{sec:wedge-twisted-factorization}
to obtain triviality for $m\leq-2$. Generative artificial intelligence
tools, specifically ChatGPT and Codex by OpenAI, were also used during
the preparation of this manuscript as auxiliary aids for language
editing, improving exposition and organization. The authors take full
responsibility for the mathematical content, references, and conclusions
of the manuscript.\\

\textbf{Data Availability.} Data sharing is not applicable to this article as no datasets were generated during the current study.\\

\textbf{Conflict of interest.} The authors declare that there is no  conflict of interest.\\

\textbf{Acknowledgments.} The second author was supported by CNPq/Brazil Grant 409513/2023-7.
\bibliographystyle{plain}
\bibliography{references}
\end{document}